\documentclass[10pt]{article}

\usepackage{authblk}

\usepackage{geometry}
\usepackage{graphicx} 
\usepackage{amsmath}
\usepackage{amssymb}
\usepackage{amsthm}
\usepackage{mathrsfs}
\usepackage{enumitem}
\usepackage{tikz}
\usepackage{color}
\usepackage[colorlinks=true,urlcolor=blue,
citecolor=red,linkcolor=blue,linktocpage,pdfpagelabels,bookmarksnumbered,bookmarksopen]{hyperref}
\usepackage{dirtytalk}
\usepackage[hang,flushmargin,stable]{footmisc} 
\usepackage{dsfont}

\usepackage{ulem}

\numberwithin{equation}{section}

\newtheorem{theorem}{Theorem}[section]
\newtheorem{corollary}[theorem]{Corollary}
\newtheorem{lemma}[theorem]{Lemma}
\newtheorem{proposition}[theorem]{Proposition}
\newtheorem*{theorem*}{Theorem}
\newtheorem*{lemma*}{Lemma}
\newtheorem*{definition*}{Definition}
\theoremstyle{definition}
\newtheorem{example}{Example}
\newtheorem{remark}{Remark}
\newtheorem{definition}[theorem]{Definition}

\newtheorem{assumption}[theorem]{Assumption}
\newtheorem*{ansatz}{Ansatz}

\renewcommand\refname{Bibliography}

\makeatletter

\newcommand\backmatter{%
  \if@openright
    \cleardoublepage
  \else
    \clearpage
  \fi
   }

\makeatother

\newcommand{\gt}{{\widetilde g}}
\newcommand{\Ht}{{\widetilde H}}

\newcommand{\wt}{{\widetilde w}}

\newcommand{\Xt}{{\widetilde X}}

\newcommand{\Lt}{{\widetilde L}}

\newcommand{\mb}{{\overline{m}}}
\newcommand{\alb}{{\overline{\a}}}
\newcommand{\ab}{{\overline{a}}}
\newcommand{\Ab}{{\overline{A}}}
\newcommand{\xb}{{\overline{x}}}
\newcommand{\Xb}{{\overline{X}}}
\newcommand{\Yb}{{\overline{Y}}}

\newcommand{\Xh}{{\widehat{X}}}
\newcommand{\Yh}{{\widehat{Y}}}
\newcommand{\Ah}{{\widehat{A}}}

\newcommand{\Lip}{\textnormal{Lip}}

\def\a{\alpha}
\def\b{\beta}
\def\d{\delta}
\def\D{\Delta}
\def\g{\gamma}
\def\k{\kappa}
\def\l{\lambda}

\def\m{\mu}

\def\na{\nabla}
\def\s{\sigma}

\def\th{\theta}
\def\t{\tau}
\def\e{\varepsilon}
\def\Om{\Omega}

\def\i{\infty}

\newcommand{\R}{{\mathbb R}}

\newcommand{\N}{\mathbb{N}}

\newcommand{\Rd}{{\mathbb R^d}}

\newcommand{\E}{{\mathbb E}}
\newcommand{\PR}{{\mathbb P}}

\newcommand{\A}{\mathcal{A}}

\newcounter{hyp}
\newcommand\norm[1]{\lVert#1\rVert}

\newcommand{\p}{\partial}

\newcommand{\spt}{{\rm{spt}}}

\renewcommand{\ae}{{\rm{a.e.}}}

\newcommand{\op}{{\rm{op}}}

\newcommand{\sF}{\mathscr{F}}
\newcommand{\sL}{\mathscr{L}}

\newcommand{\prob}{{\mathscr P}}

\newcommand{\Leb}{\mathcal{L}}

\begin{document}

\title{Distributed Equilibria for $N$-Player Differential Games with Interaction through Controls: Existence, Uniqueness and Large $N$ Limit}

\date{\today}

\author[1]{Hei Jie Lam \& Alp\'ar R. M\'esz\'aros \thanks{Department of Mathematical Sciences, Durham University, Durham, DH1 3LE, UK.\\ Email addresses: hei.j.lam@durham.ac.uk \& alpar.r.meszaros@durham.ac.uk}}

\renewcommand\refname{Bibliography}

\maketitle

\begin{abstract}
    We establish the existence and uniqueness of distributed equilibria to possibly nonsymmetric $N$-player differential games with interactions through controls under displacement semimonotonicity assumptions. Surprisingly, the nonseparable framework of the running cost combined with the character of distributed equilibria leads to a set of consistency/fixed point relations different in nature from the ones for open- and closed-loop equilibria investigated in the recent work of Jackson and the second author in \cite{jackson2025quantitativeconvergencedisplacementmonotone}. In the symmetric setting, we establish quantitative convergence results for the $N$-player games toward the corresponding Mean Field Games of Controls (MFGC), when $N\to+\infty$. Our approach applies to both idiosyncratic noise–driven models and fully deterministic frameworks. In particular, for deterministic models distributed equilibria correspond to open-loop equilibria, and our work seems to be the first in the literature to provide existence and uniqueness of these equilibria and prove the large $N$ convergence in the MFGC setting. The sharpness of the imposed assumptions is discussed via a set of explicitly computable examples in the linear quadratic setting.
\end{abstract}

\section{Introduction}

\medskip

\subsection*{Formulation of the $N$-player games.} 

Fix $d,N\in\N,\ T>0,\ \b\geq 0$. Let $(\Om,\sF,(\sF_t)_{t\in[0,T]},\PR)$ be a filtered probability space satisfying the usual conditions supporting $N+1$ mutually independent Brownian motions $B,\ (B^i)_{i=1,\ldots,N}$. Furthermore, we assume $\Om$ to be a Polish space, $\sF$ to be its Borel $\s$-algebra and $\PR$ to be atomless. Given $N$ mutually independent random variables $\xi^i$, $i=1,\ldots,N$, let $(\sF^i_t)_{t\in[0,T]}$ denote the augmented filtration generated by $B^i$ and $\xi^i$. For $p\ge 1$, we denote by $\prob_p(\Rd)$ the set of Borel probability measure supported on $\Rd$ with finite $p$-order moments. This set is equipped with the classical Wasserstein metric $W_p$.

\medskip

In the distributed formulation the $i$-th player chooses as controls feedback functions $\a^i:[0,T]\times\Rd\rightarrow\Rd$, and the corresponding state process $(X^i_t)_{t\in[0,T]}$ evolves according to the state equation

\begin{equation*}
    \begin{cases}
    dX^i_t=\a^i(t,X^i_t) dt+\sqrt{2\b} dB^i_t, &\textnormal{in } (0,T),\\
    X^i_0=\xi^i.
\end{cases}\end{equation*}

 Notice that $X^i_t$ is independent of $X^j_t$ for $j\neq i.$

\begin{definition}
    A feedback function $\a:[0,T]\times\Rd\rightarrow\Rd$ is an {\it admissible distributed control} whenever it is a measurable function such that the SDE above is (strongly) wellposed, in which case we write $\a\in\A^{dist}$. Throughout this paper we mean admissible feedback controls whenever we write controls.
\end{definition}

The notations $X_t^{-i}=(X_t^j)_{j\neq i}=(X_t^1,\ldots,X_t^{i-1},X_t^{i+1},\ldots,X_t^N)$ and $\a^{-i}(t,X^{-i}_t)=(\a^j(t,X^j_t))_{j\neq i}$ will be used throughout this paper.

\medskip

Given an $N$-tuple of controls $\a=(\a^1,\dots,\a^N)\in(\A^{dist})^{N},$  the cost functional of the $i$-th player is given by

\[J^i(\a^i;\a^{-i}):=\E\left[\int^T_0L^i\bigl(X^i_t,\a^i(t,X^i_t),X^{-i}_t,\a^{-i}(t,X^{-i}_t)\bigr)\ dt+g^i(X^i_T,X^{-i}_T)\right]\]

where the functions $L^i:\Rd\times\Rd\times(\Rd)^{N-1}\times(\Rd)^{N-1}\rightarrow\R$, $g^i:\Rd\times(\Rd)^{N-1}\rightarrow\R$ are the running cost and final cost for the $i$-th player, respectively.

\begin{definition}
    An $N$-tuple of controls $\a_*=(\a^1_*,\dots,\a^N_*)\in(\A^{dist})^{N}$ defines a {\it distributed Nash equilibrium} if the following holds for all $i=1,\ldots,N$

\begin{equation}\label{equilibrium condition}
J^i(\a_*)\leq J^i(u;\a_*^{-i})\ \ \ \forall u\in\A^{dist},
\end{equation}
where we have used the notation $(u;\a_*^{-i}) := (\a_*^1,\dots,\a_*^{i-1},u,\a_*^{i+1},\dots,\a_*^N).$

\end{definition}

We define the \say{lifted} functions $\Lt^i:\Rd\times\Rd\times(\prob_2(\Rd))^{N-1}\times (C_b(\Rd))^{N-1}\rightarrow\R$, $\gt^i:\Rd\times(\prob_2(\Rd))^{N-1}\rightarrow\R$ by

\[\Lt^i(x^i,a^i,m^{-i},f^{-i}):=\int_{(\Rd)^{N-1}}L^i\bigl(x^i,a^i,x^{-i},f^{-i}(x^{-i})\bigr)\ dm^{-i}(x^{-i}),\]

\[\gt^i(x^i,m^{-i}):=\int_{(\Rd)^{N-1}}g^i\bigl(x^i,x^{-i}\bigr)\ dm^{-i}(x^{-i}).\]

where $dm^{-i}(x^{-i})=\prod_{j\neq i}dm^j(x^j)$.

\medskip

Considering the stochastic optimal control problem faced by the $i$-th player, we formally derive the following PDE system 
\begin{equation}\label{DNash}
\left\{
    \begin{array}{ll}
    -\p_tw^i(t,x)-\b\D w^i(t,x)+\Ht^i(x,D_xw^i(t,x),m^{-i}_t,\a^{-i}_t)=0, & {\rm in}\ (0,T)\times\Rd,\\[3pt]
    \p_t m^i_t-\b\D m_t^i +\na\cdot(m^i_t \a^i_t)=0, & {\rm in}\ (0,T)\times\Rd,\\[3pt]
    \a^i(t,x)=-D_p\Ht^i(x,D_xw^i(t,x),m^{-i}_t,\a^{-i}_t)), & {\rm in}\ (0,T)\times\Rd,\\[3pt]
    w^i(T,x)=\gt^i(x,m^{-i}_T),\ m^i_0=\m^i, & {\rm in}\ \Rd.
\end{array}
\right.
\end{equation}

where $\m^i=\Leb(\xi^i)=\xi^i_\sharp \mathbb{P}$ and $\Ht^i:\Rd\times\Rd\times(\prob_2(\Rd))^{N-1}\times (C_b(\Rd))^{N-1}\rightarrow\R$ is the Hamiltonian associated to $\Lt^i$ defined by
\[\Ht^i(x,p,m^{-i},f^{-i})=\max_{a\in\Rd}\{-p\cdot a-\Lt^i(x,a,m^{-i},f^{-i})\}.\]
\begin{definition}\label{def classical solution}
    We say that a triplet $(w,m,\a)=(w^i,m^i,\a^i)_{i=1,\ldots,N}$ is a  solution of the PDE system (\ref{DNash}) if $w^i\in C^{1,2}([0,T]\times\Rd)$ ($C^1([0,T]\times\Rd)$ if $\b=0$) are classical solutions to the Hamilton--Jacobi--Bellman (HJB) equations, $m^i\in C([0,T];\prob_1(\Rd))$ are distributional solutions to the Kolmogorov--Fokker--Planck (KFP) equations and $\a^i\in C([0,T];W^{1,\i}(\Rd))$ ($C([0,T];\Lip(\Rd))$ if $\b=0$) are vector fields that are Lipschitz continuous in spacial variable satisfying the third \say{consistency} equations. 
\end{definition}
By classical verification arguments, one can show that under suitable convexity assumptions a solution $(w,m,\a)$ to (\ref{DNash}) defines an $N$-tuple of distributed equilibrium controls $\a\in(\A^{dist})^N$. For the reader's convenience we include a proof of this result in Theorem \ref{Verification thm}.
\begin{remark}\label{trivial consistency relation}
    The consistency relation 
    \[\a^i(t,x)=-D_p\Ht^i(x,D_xw^i(t,x),m^{-i}_t,\a^{-i}_t))\]
    becomes trivial if $D_p\Ht^i$ does not depend on the last variable. This happens for example when the running costs are of the form
    \begin{equation}\label{eq:semi-sep-con}
    L^i(x^i,a^i,x^{-i},a^{-i})=L^i_1(x^i,a^i,x^{-i})+L^i_2(x^i,x^{-i},a^{-i}).
    \end{equation}
 \end{remark}
  This motivates the following definition.
  
 \begin{definition}\label{def:sep}
    We say that the running costs are {\it semi-separable in the control distribution}, if they satisfy \eqref{eq:semi-sep-con} for some $L^i_1:\Rd\times\Rd\times(\Rd)^{N-1}\to\R$, $L^i_2:\Rd\times(\Rd)^{N-1}\times(\Rd)^{N-1}\to\R.$

\medskip

    If the running costs can be further separated into the form 
    \[L^i(x^i,a^i,x^{-i},a^{-i})=L^i_1(x^i,a^i)+L^i_2(x^{-i})+L^i_3(a^{-i}),\]
    for some $L^i_1:\Rd\times\Rd\to\R$, $L^i_2,L^i_3:(\Rd)^{N-1}\to\R,$
    we say that the running costs are {\it fully separable}, in which case we can simplify the expressions of the Hamiltonians associated to the lifted running costs as
    \[\Ht^i(x,p,m^{-i},f^{-i})=H^i(x,p)-\Lt^i_2(m^{-i})-\Lt^i_3(\a^{-i}),\]
    where $H^i$ is the {classical} Hamiltonian associated to $L^i_1$, i.e.
    \[H^i(x,p)=\max_{a\in\Rd}\{-p\cdot a-L^i_1(x,a)\}.\]
Another form of separability for the running costs is
\[L^i(x^i,a^i,x^{-i},a^{-i})=L^i_1(x^i,a^i,a^{-i})+L^i_2(x^i,x^{-i},a^{-i}),\]
for some $L^i_1:\Rd\times\Rd\times(\Rd)^{N-1}\to\R$, $L^i_2:\Rd\times(\Rd)^{N-1}\times(\Rd)^{N-1}\to\R$, and 
we say that this form of running costs is {\it semi-separable in the state distribution}.
\end{definition}


It is well known that the Pontrayagin maximum principle (PMP) is a necessary condition of optimality (e.g. \cite[Theorem 3.3]{yong2012stochastic}, \cite[Theorem 4.1]{alparmouMFGunderDmonotone}, \cite[Section 1.1.3]{cirant2025nonasymptoticapproachstochasticdifferential}), given $(w,m,\a)$ a classical solution to (\ref{DNash}), we can associate a (strong) solution to the FBSDE called the PMP system,
\begin{equation}\label{PMP}
    \begin{cases}
        X^i_t=\displaystyle\xi^i+\int^t_0\a^i(s,X^i_s)\ ds+\sqrt{2\b} B^i_t, & t\in [0,T],\ \\[5pt]
        Y^i_t=\displaystyle D_x\gt^i(X^i_T,m^{-i}_T)-\int^T_tD_x\Ht^i(X^i_s,Y^i_s,m^{-i}_s,\a^{-i}_s)\ ds-\int^T_tZ^i_s\ dB^i_s , & t\in [0,T],   \end{cases}
\end{equation}
with $Y^i_t=D_xw^i(t,X^i_t)$ and $Z^i_t=\sqrt{2\b}D^2_{xx}w^i(t,X^i_t)$ for all $t\in[0,T]$.

\medskip

Let us also we briefly mention the fully deterministic setting, that is, $\b=0$ and $\xi^i=z^i$ for some $z^i\in\Rd$, $i\in\{1,\dots,N\}$. In this case the system (\ref{DNash}) simplifies to
\begin{equation}\label{DNash deterministic}
\left\{
    \begin{array}{ll}
        -\p_tw^i(t,x)+H^i(x,D_xw^i(t,x),X^{-i}_t,\a^{-i}(t,X^{-i}_t))=0, & {\rm in}\ (0,T)\times\Rd,\\[3pt]
        \dot{X}^i_t=\a^i(t,X^i_t), & {\rm in}\ (0,T),\\ [3pt]
    \a^i(t,x)=-D_pH^i(x,D_xw^i(t,x),X^{-i}_t,\a^{-i}(t,X^{-i}_t)), & {\rm in}\ (0,T)\times\Rd,\\[3pt]
    w^i(T,x)=g^i(x,X^{-i}_T),\ X^i_0=z^i, & {\rm in}\ \Rd,
    \end{array}
   \right.
\end{equation}

where the Hamiltonian is defined in its usual form
\[H^i(x,p,y^{-i},A^{-i})=\max_{a\in\Rd}\{-p\cdot a-L^i(x,a,y^{-i},A^{-i})\}.\]
In fact this is simply a special case of the lifted Hamiltonians evaluated at the flows of dirac measures $\d_{X^i_t}$. Similarly, the PMP \eqref{PMP} also simplifies to the FBODE

\begin{equation}\label{deterministic PMP}
    \begin{cases}
        X^i_t=\displaystyle z^i+\int^t_0\a^i(s,X^i_s)\ ds,  & t\in [0,T],\\[5pt]
        Y^i_t=\displaystyle D_xg^i(X^i_T,X^{-i}_T)-\int^T_tD_xH^i(X^i_s,Y^i_s,X^{-i}_s,\a^{-i}(s,X^{-i}_s))\ ds,  & t\in [0,T].  
    \end{cases}
\end{equation}

\begin{remark}\label{remark 1}
    Notice that \eqref{equilibrium condition} is equivalent to if we replace $\A^{dist}$ by the set of all $(\sF^i_t)$-progressively measurable processes. In particular, in a fully deterministic game, a distributed equilibrium is an open-loop equilibrium, as the underlying filtered probability space and filtrations generated by different Brownian motions no longer play a role. For a more detailed explanation, we refer to Section \ref{discussion}.
\end{remark}

\subsection*{Formulation of the MFGCs.}
We describe now the setting in the mean field limit. First we introduce the cost functions, which will be particularised as follows. For all $i=1,\ldots,N$, let 
\begin{equation}\label{MF cost functions}
L^i(x^i,a^i,x^{-i},a^{-i}):=L\left(x,a,m^{N,-i}_{x,a}\right),\ g^i(x^i,x^{-i}):=g\left(x,m^{N,-i}_x\right),\end{equation}
where $L:\Rd\times\Rd\times\prob(\R^{d}\times\Rd)\rightarrow\R$, $g:\Rd\times\prob(\Rd)\rightarrow\R$ and for $z\in\R^k$ we set
\[m^{N,-i}_z:=\frac{1}{N-1}\sum_{j\neq i}\d_{z^j}.\]
Given a flow of measures $(\m_t)_{t\in[0,T]}\in C([0,T];\prob_2(\R^{d}\times\Rd))$ representing the joint evolution of an aggregate of players and their respective controls, the state equation of a representative player is given by
\[
\left\{
\begin{array}{ll}
    dX_t=\a_t dt+\sqrt{2\b} dB_t, & t\in (0,T), \\
    X_0=\xi,
\end{array}
\right.
\]
where admissible controls are all $(\sF_t)$-progressively measurable processes, $\xi$ is a random variable such that its law $\Leb(\xi)=m_0=(\pi_1)_\sharp\mu_0$ is the first marginal of $\m_0$.

\medskip

Let us denote by $m_t:=(\pi_1)_\sharp\m_t$  the first marginal of the flow of measures $(\m_t)_{t\in[0,T]}.$ The cost functional of a representative player, associated to the state dynamics and the given flow of measures, is given by
\[J(\a):=\E\left[\int^T_0L\bigl(X_t,\a_t,\m_t)\bigr)\ dt+g(X_T,m_T)\right].\]

\medskip

Denoting a minimiser by $\a^*=(\a^*_t)_{t\in[0,T]}$ and the corresponding optimal state process as $X^*=(X^*_t)_{t\in[0,T]}$, we say the flow $(\m_t)_{t\in[0,T]}$ forms a {\it mean field equilibrium} if 
\begin{equation}\label{compatibility condition}
\Leb(X^*_t,\a^*_t)=\m_t,\ \ \ \ \forall t\in[0,T].\end{equation}
By considering the optimisation problem faced by the representative player, one finds that mean field equilibria are characterised by the PDE system (c.f. \cite{jackson2025quantitativeconvergencedisplacementmonotone,Kobeissi04032022})
\begin{equation}\label{MFG}
\left\{    
    \begin{array}{ll}
        -\p_tv(t,x)-\b\D v(t,x)+H(x,D_xv(t,x),\m_t)=0, & {\rm in}\ (0,T)\times\Rd,\\[3pt]
        \p_tm_t-\b\D m_t-\na\cdot(m_tD_p H(x,D_xv,\m_t))=0, & {\rm in}\ (0,T)\times\Rd,\\[3pt]
        \m_t=(\textnormal{Id},-D_pH(\cdot,D_xv(t,\cdot),\m_t))_\sharp m_t, & {\rm in}\  [0,T],\\[3pt]
        v(T,x)=g(x,m_T),\ m_0=\Leb(\xi), & {\rm in}\ \Rd,\\
    \end{array}
\right.
\end{equation}
where the Hamiltonian $H:\Rd\times\Rd\times\prob_2(\Rd\times\Rd)\rightarrow\R$ is defined as
\[H(x,p,\m)=\max_{a\in\Rd}\{-p\cdot a-L(x,a,\m)\}.\]

\begin{definition}
We say that a pair $(v,\m)$ is a solution of the PDE system (\ref{MFG}) if $v\in C^{1,2}([0,T]\times\Rd)$ ($C^1$ if $\b=0$) is a classical solution to the HJB equation, $m\in C([0,T];\prob_2(\Rd))$ is a distributional solution to the KFP equation and $\m\in C([0,T];\prob_2(\Rd\times\Rd))$ satisfies the third consistency equation.    
\end{definition}

\begin{remark}\label{trivial consistency relation 2}
    As noted also in \cite{jackson2025quantitativeconvergencedisplacementmonotone}, the consistency relation 
    \[\m_t=(\textnormal{Id},-D_pH(\cdot,D_xv(t,\cdot),\m_t))_\sharp m_t\]
    becomes trivial if $D_pH$ does not depend on the second marginal of the measure variable. We characterise the notion of separability of the running costs in the mean field case by the notion of separability of their finite dimensional projections (c.f. \eqref{MF cost functions} and Remark \ref{trivial consistency relation} and Definition \ref{def:sep}).

    \medskip

    Let us note that the semi-separable in state distribution case allows for Lasry--Lions monotonicity to be employed (c.f. \cite[Page 29]{cardaliaguetMFG}). 
\end{remark}

As in the case of $N$-player games, given a solution $(v,\m)$ to (\ref{MFG}) we can associate a solution to the PMP  
\begin{equation}\label{PMPMFG}
 \left\{   
    \begin{array}{ll}
      X_t=\displaystyle\xi-\int^t_0 D_pH(X_s,Y_s,\m_s) \ ds+\sqrt{2\b} B_t, & t\in [0,T],\\
        Y_t=\displaystyle D_xg(X_T,m_T)-\int^T_tD_xH(X_s,Y_s,\m_s)\ ds-\int^T_tZ_s\ dB_s, & t\in [0,T],\\
        \m_t=\displaystyle\Leb(X_t,-D_pH(X_t,Y_t,\m_t)), & t\in [0,T],
    \end{array}
 \right.
\end{equation}
with $Y_t=D_xv(t,X_t)$ and $Z_t=\sqrt{2\b}D^2_{xx}v(t,X_t)$ for all $t\in[0,T]$.

\medskip

In order to speak about the convergence problem, for each $N\in\N,$ we also need to consider $N$ i.i.d. copies of mean field equilibrium trajectories (\ref{PMPMFG})
\begin{equation}\label{iidPMPMFG}
 \left\{  
    \begin{array}{ll}
        \Xh^i_t=\displaystyle\widehat{\xi}^i-\int^t_0 D_pH(\Xh^i_s,\Yh^i_s,\m_s) \ ds+\sqrt{2\b} B^i_t,\ & t\in[0,T],\\
        \Yh^i_t=\displaystyle D_xg(\Xh^i_T,m_T)-\int^T_tD_xH(\Xh^i_s,\Yh^i_s,\m_s)\ ds-\int^T_t\widehat{Z}^i_s\ dB^i_s, \ & t\in[0,T],
    \end{array}
\right.
\end{equation}
where $(\widehat{\xi}^i)_{i=1,\ldots,N}$ are i.i.d. random variables with common law $\Leb(\widehat{\xi}^i)=m_0$.

\subsection*{Summary of our main results.}
We will summarise our main results informally, for precise statements we refer to later sections. The following statements are to be understood under our standing assumptions that we will specify later.  In particular, our results will impose displacement semimonotonicity assumptions. For  displacement semimonotone settings in MFG --- both with and without interactions through controls --- we refer to the non-exhaustive list of works \cite{GangboMeszarosPotentialMasterEqn, GangboMeszarosMouZhangMasterEqn, BanMesMou, BanMes:25-fms, CirantRedaelli, cirant2025nonasymptoticapproachstochasticdifferential, carmonadelarue1, jacksontangpi, JacMes:26, jackson2025quantitativeconvergencedisplacementmonotone}.

\medskip

We emphasise that in these following results we allow for all nonnegative values of $\b\geq 0$.
\begin{theorem}
   Under our standing assumptions, there exists a unique classical solution to the system \eqref{DNash}, i.e. distributed equilibria to the $N$-player games exist and are unique. Moreover, let $(w,m,\a),(\overline{w},\mb,\alb)$ be two classical solutions to \eqref{DNash} with initial measures $\m=(\mu^1,\dots,\mu^N),\overline{\m}=(\overline{\m}^1,\dots,\overline{\m}^N)$ respectively. Then there exists a constant $C>0$ such that
    \[\sup_{t\in[0,T]}\sum_{i=1}^NW_2(m^i_t,\mb^i_t)\leq C\sum^N_{i=1}W_2(\m^i,\overline{\m}^i).\]
\end{theorem}


\begin{theorem}
    Under our standing assumptions, the system \eqref{MFG} is wellposed, i.e. MF equilibria to MFGCs exist and are unique. 
\end{theorem}

\begin{theorem}
     Suppose that our standing assumptions take place.
     \begin{itemize}[label=\raisebox{0.25ex}{\tiny$\bullet$}]
         \item Let $(\xi^i)_{i=1,\ldots,N}$ be $L^q$-independent random variables;
         \item Let  $(\widehat{\xi}^i)_{i=1,\ldots,N}$ be a $L^q$-i.i.d. random variables, with common law $\Leb(\widehat{\xi}^i)=m_0\in\prob_q(\Rd)$, for some $q>2$, $q\notin\{4,\frac{d}{d-2}\}$, $i=1,\ldots,N$.
     \end{itemize}
     Then, there exists a constant $C>0$ independent of $N\in\N$ such that
    \[\frac{1}{N}\sum^N_{i=1}\E\Biggl[\sup_{t\in[0,T]}|X^i_t-\Xh^i_t|^2+\int^T_0|A^i_t-\Ah^i|^2\ dt\Biggr]\leq C(K(N) +r_{d,q}(N)),\]
    \[\max_{i=1,\ldots,N}\norm{w^i-v}_\infty\leq C\left(K(N)+r_{d,q}(N)\right)^\frac{1}{2}\]
    where
    \begin{itemize}[label=\raisebox{0.25ex}{\tiny$\bullet$}]
        \item $X=(X^1,\dots,X^N)$ is the first component of the solution to \eqref{PMP} with initial condition $\xi = (\xi^1,\dots,\xi^N)$;
        \item $\Xh=(\Xh^1,\dots,\Xh^N)$ is the first component of the solution to \eqref{iidPMPMFG} with initial condition $\widehat{\xi} = (\widehat{\xi}^1,\dots,\widehat{\xi}^N)$;
        \item $(w,m,\a)$ is the solution to the $N$-player system \eqref{DNash} with $m^i_0=\Leb(\xi^i)\in \prob_q(\Rd)$;
        \item $(v,\m)$ is the solution to the MFG system \eqref{MFG};
        \item $A,\Ah$ are the corresponding control processes given by $A^i_t=\a^i(t,X^i_t)$ and $ \Ah_t^i=-D_pH(\Xh^i_t,\Yh^i_t,\m_t)$ 
    \end{itemize} and 
    \[K(N):=\frac{1}{N}\sum_{i=1}^N W_2^2(m_0,\Leb(\xi^i)),\]
    \[r_{d,q}(N)=\begin{cases}
    N^{-\frac{1}{2}}+N^{-\frac{q-2}{q}}\ &d<4\\
    N^{-\frac{1}{2}}\log(1+N)+N^{-\frac{q-2}{q}}\ &d=4\\
    N^{-\frac{2}{d}}+N^{-\frac{q-2}{q}}\ &d>4.
\end{cases}\]
Furthermore, in the case $\b=0$ and $m^i_0=\d_{z^i}$, i.e. the case of \eqref{DNash deterministic}, we have also     \[\max_{i=1,\ldots,N}\norm{D_xw^i-D_xv}_\infty\leq C\left(K(N)+r_{d,q}(N)\right)^\frac{1}{2}.\]
\end{theorem}

\subsection*{Literature review related to our main results.}

\textit{MFGs.} The theory of Mean Field Games (MFGs) was introduced independently by Lasry--Lions \cite{MFGlionslasry} and Huang--Malham\'e--Caines \cite{MFGhuangcaines} to provide an approximation of $N$-player games when the number of player $N$ is very large and the curse of dimensionality renders the computation of Nash equilibria intractable for $N$-player games. From the PDE perspective, a typical MFG  can be characterised by solutions to the PDE system

\begin{equation}\label{standard MFG}
\left\{    
    \begin{array}{ll}
        -\p_tv-\b\D v+H(x,D_xv,m_t)=0, & {\rm in}\ (0,T)\times\Rd,\\[3pt]
        \p_tm_t-\b\D m_t-\na\cdot(m_tD_p H(x,D_xv,m_t))=0, & {\rm in}\ (0,T)\times\Rd,\\[3pt]
        v(T,x)=g(x,m_T),\ m_0=\Leb(\xi), & {\rm in}\ \Rd,
    \end{array}
\right.
\end{equation}
where $v:(0,T)\times\R^d\to\R$ denotes the value function of a typical agent, and $(m_t)_{t\in[0,T]}$ represents the flow of the distribution of the states of the agents.

\medskip

Since the inception of the theory one of the key problems is to rigorously justify the approximation of $N$-player games by corresponding MFGs. There are in general two methods of establishing this justification, the first of which initiated in \cite{MFGhuangcaines} is the construction of approximate Nash equilibria for the $N$-player games from the MF equilibrium strategy, the second of which, typically known as the {\it convergence problem} in the MFG literature, is to prove directly that Nash equilibria of $N$-player games converge to the MF equilibria. 

\medskip

As is well known, the notion of equilibria used in the $N$-player games impacts greatly the difficulty of the convergence problem. The qualitative convergence of open-loop equilibria has been established in \cite{Fischer2014convergenceofopenloopeq,LackerConvOpenLoop}, while for closed-loop equlibria \cite{LackerConvClosedLoop} proves that limit points of closed-loop equilibria are \say{weak MFG equilibria} and \cite{CirantRedaelli} proves via a priori estimates on the so-called \say{Nash system} the convergence of closed-loop equilibria to strong equilibria under monotonicity conditions. A qualitative convergence result for distributed equilibria in a separable cost framework was given already in the seminal work \cite{MFGlionslasry}; we note, however, that the authors did not emphasise greatly the terminology of \say{distributed equilibria}. Another notable paper is \cite{BardiLQMFG}, where explicit solutions for linear quadratic (LQ) MFGs and distributed equilibria for LQ $N$-player games were computed in the ergodic setting and the corresponding convergence results were shown explicitly; we note that the author did not use the terminology of distributed equilibria. The terminology of distributed controls may cause some confusion to arise, as it is used not only in the specific setting of $N$-player games/control problems (and their convergence to MFG/Mean Field Control) as it was noted in \cite{JacksonDistribuedMeanFieldControl} that \say{the term {\it distributed} is used broadly in the control theory literature, with a variety of different meanings.} 

\medskip

Quantitative convergence results, i.e. results with an explicit rate of convergence, are obtained for closed-loop equilibria in the seminal work \cite{masterequation}, via the \say{master equation}. For open-loop equilibria, quantitative convergence is proven in \cite{LauriereConvMFGC,jacksontangpi} with the method of \say{forward-backward propagation of chaos}, notably the former paper allows interaction through controls of other players and the latter allows controlled volatility. More precise quantitative convergence results were obtained in \cite{cirant2025nonasymptoticapproachstochasticdifferential}, where the authors further quantify the difference between open/closed-loop/distributed equilibria in a separable cost framework, without appealing to the master equation. The recent contribution \cite{jackson2025quantitativeconvergencedisplacementmonotone} quantifies the difference between open/closed-loop equilibria in a nonseparable cost framework, where in particular interactions through controls were also allowed. These methods were then pushed forward in \cite{JacMes:26} to build a solution to the corresponding master equations.

\medskip

\textit{MFGCs.} There has recently been an increasing interest in the theory of \say{Mean Field Games of controls} (MFGC) (also called extended MFGs), where compared to the initially introduced models in the theory, interactions are also via distributions of controls of players, in addition to via distributions of states of players. MFGCs were first introduced in \cite{Gomes2013EMFG2} under the terminology of extended MFGs. The theory of MFGCs often has wider applicability in economics modelling, e.g. \cite{TradingMFGC,RenweablesMFGC,CournotBertrandMFGC,AlasseurMFGCSmartGrids}. 

\medskip

The wellposedness of MFGCs with nonseparable cost functions (in the control distribution) is more difficult, as in addition to the usual HJB--KFP equations in the standard MFG system \eqref{standard MFG}, they are accompanied by the additional consistency relation 
\[\m_t=(\textnormal{Id},-D_pH(\cdot,D_xv(t,\cdot),\m_t))_\sharp m_t.\]
This is a fixed point relation set on infinite dimensional space $\prob(\Rd\times\Rd)$, nonetheless this has been well studied by now.

\medskip

The first wellposedness theory for the PDE system characterising the MFGC on the state space $\Rd$ was obtained in \cite{KobeissiMFGC2}, in the setting of semi-separability in state distribution. The author imposes a monotonicity condition of Lasry--Lions type (this has been observed in \cite{alparmouMFGunderDmonotone,GangboMeszarosPotentialMasterEqn} to be a completely different regime with the regime of displacement monotonicity; however, under suitable second derivative bounds, the Lasry--Lions monotonicity implies displacement semimonotonicity, c.f. \cite[Remark 2.8]{GangboMeszarosMouZhangMasterEqn}). On the torus $\mathbb{T}^d$, \cite{PotentialMFGC} proved wellposedness in a potential game setting and \cite{Kobeissi04032022} proved various general wellposedness results under smallness conditions.  All the references cited above rely on the \textit{uniform parabolic} setting coming from the stochastic framework. 

\medskip

\medskip

In the deterministic setting, \cite{GomesDeterministicMFGC} proves wellposedness for MFGCs, with uniqueness coming from a generalised Lasry–Lions type monotonicity. In addition, examples of explicitly solved MFGCs were given and the master equation was also discussed. Wellposedness for deterministic MFGCs were also obtained in \cite{graber2024newuniquenessresultsmean} for a model where the dependence on the distribution of state and controls is finite dimensional, the authors used a monotonicity which is different from both Lasry–Lions and displacement monotonicity (introduced in \cite{GraMes:23, GraMes:24}). 

\medskip

We would also like to mention that a wellposedness theorem for the FBSDE systems of Pontryagin type with nondegenerate idiosyncratic noise (hence also strong MFG equilibria in the presence of sufficient convexity assumptions) can be found in the appendix of \cite{jackson2025quantitativeconvergencedisplacementmonotone} in the setting of uniformly parabolic displacement semimontone MFGCs. We extend this result by showing that this naturally leads to a wellposedness theorem for MFGC PDE systems. Moreover, in the current paper we provide two avenues to obtain wellposedness of deterministic MFGCs in the setting of displacement semimonotonicity, the second of which is provided as a by-product of a vanishing viscosity result showing the convergence of uniform parabolic MFGC PDE systems to the corresponding first order MFGC PDE systems.

\medskip

Some other recent works on MFGCs include the paper \cite{graberMFGCFractionLaplacian}, where states of players can \say{jump}, leading to a fractional Laplacian instead of the usual Laplacian in the MFGC system and the paper \cite{graber2025MFGCBoundaryCOnditions}, where arbitrary (convex resp.) bounded state space was considered, complemented by Dirichlet (Neumann resp.) boundary conditions.  

\medskip

Probabilistic treatments of open- and closed-loop equilibria associated to MFGC have been developed in the series of papers \cite{ Dje:23-esaim, Dje:23-aap,  PosTan}. Instead of PDE systems, these papers focus on probabilistic aspects of MFG theory (c.f. \cite{carmonadelarue1}) extended to the MFGC setting (such as weak formulations\footnote{Weak as in the probabilistic sense, not PDE sense.} and BSDEs characterisations). In the former two papers, the author constructs approximate open- and closed-loop Nash equilibria and shows that these arise as a limit of so-called \say{measure-valued MFG equilirbia}, which is an extension of classical weak MFG equilibria to the MFGC setting. He then proves these to exist (but not necessarily unique). In the latter paper, the authors characterise weak MFG equilibria with a BSDE associated to the dynamic programming principle, for which they establish a well-posedness theory. They then further prove convergence for these weak equilibria. The aforementioned papers work under assumptions that guarantee the solvability of the consistency relation.

\medskip

Note that in the main text of the current paper we will no longer specify the terminology MFGCs and simply refer to MFGs. 

\medskip

\textit{N-Player Games.} As in their mean field counterparts, $N$-player games with interactions through controls are used to model problems arising in economics, we refer to the papers \cite{AlasseurMFGCSmartGrids,HarrisCournotNPlayerGames,Ledvina2010DynamicBO} for various interesting applications. 

\medskip

 However, to the best of our knowledge, an analytical treatment of $N$-player games with interactions through controls (that leads to a set of \say{non-trivial} consistency relations) 
 has so far only been done in \cite{jackson2025quantitativeconvergencedisplacementmonotone}, where the FBSDE/PDE systems describing open/closed-loop equilibria were derived. Through the analysis of the consistency relation, the authors in \cite{jackson2025quantitativeconvergencedisplacementmonotone} prove the existence of a \say{fixed point mapping} which, once proven to exist, effectively simplifies the equations to a somewhat usual form seen as in the literature of differential games without interaction through controls. Wellposedness results for open-loop and mean field equilibria are given, while closed-loop equilibria were assumed to be well-posed. The authors focused on the quantitative convergence problem for $N$-player games, the difference between open-loop equilibria for $N$-player games and mean field equilibria was first quantified and shown to vanish as $N\rightarrow\i$, the difference between open/closed-loop equilibria for $N$-player games was next quantified and also shown to vanish as $N\rightarrow\i$, consequently a quantitative convergence result for closed-loop equilibria was also obtained.

\medskip

The literature on an analytical treatment of PDE systems arising from distributed equilibria in $N$-player games is limited. Some initial results can be found in the seminal paper \cite{MFGlionslasry} in the setting of no interactions through controls and fully separable running costs. As we have already seen, due to the lift of the Hamiltonians to an infinite dimensional space, the analysis of these PDE systems are complicated in the setting of nonseparable running costs. To the best of our knowledge, our paper provides the first analytical treatment of these PDE systems in the setting of nonseparable running costs, with interactions through controls and nontrivial consistency relations.

\medskip

Let us emphasise that in the case of deterministic problems, there is an extra motivation for the comprehensive study of distributed/open-loop equilibria. Indeed, to the best of our knowledge, it remains a challenging open question to establish the well/ill-posedness of PDE systems associated to closed-loop equilibria, even for games with no interactions through controls (beyond linear quadratic or highly symmetric cases). Such PDE systems, usually referred to as Nash systems in the literature, degenerate to strongly coupled nonlinear systems of PDEs.

\medskip

Our paper, together with \cite{jackson2025quantitativeconvergencedisplacementmonotone}, provides a comprehensive treatment of the wellposedness and convergence problem of displacement semimonotone $N$-player games and the analysis of the corresponding mean field games both for stochastic and deterministic problems where players interact through the distribution of their controls (without controlled volatility), unconditional on the solvability of any consistency relations that arise.

\medskip

\subsection*{The description of our approach.} 
Although our results can be seen as a natural extension of the results from \cite{jackson2025quantitativeconvergencedisplacementmonotone} (towards the direction of the study of distributed equilibria, as in \cite{cirant2025nonasymptoticapproachstochasticdifferential}), a few obstacles have to be overcome and significant new ideas are needed for these.  Therefore, to tackle the difficulties arising from the distributed framework combined with the interaction through controls, we need to choose a different approach from the ones established previously in the literature.

\medskip

Firstly, unlike \cite{jackson2025quantitativeconvergencedisplacementmonotone} we do not analyse the consistency relation directly. This is due to the fact that in the distributed framework within the stochastic setting, one requires to \say{lift} the running cost to the functionals $\Lt^i$ defined over suitable infinite dimensional function spaces. This lifting procedure in general does not commute with the Legendre transform, therefore the resulting lifted Hamiltonian $\Ht^i$ is completely different from the classical Hamiltonian $H$, on which the fixed point mapping of \cite{jackson2025quantitativeconvergencedisplacementmonotone} is defined. This phenomenon is also due to the fact that the Hamiltonians/running costs in our setting are considered to be nonseparable, which was not the case in particular in \cite{cirant2025nonasymptoticapproachstochasticdifferential}.
 To bypass this obstacle, we perform a Schauder type argument on the feedback controls, instead of on the measure flow.

\medskip

Even though the deterministic setting does not require us to perform any lifting procedure, employing the fixed-point mapping of \cite{jackson2025quantitativeconvergencedisplacementmonotone} would result in coupling the HJB equations into a hardly trackable first order system, which would lead to a highly nontrivial analysis when freezing the measure flow. Therefore, we chose to follow similar Schauder type arguments, which are robust enough to work in the deterministic setting as well. An advantage of this approach --- compared to \cite{jackson2025quantitativeconvergencedisplacementmonotone} when it comes to establishing the existence of equilibria --- is that we do not need to consider $N$ to be large enough (hence $N\in\N$ can be arbitrary).

\medskip

The method of proving stability, uniqueness and quantitative convergence of equilibrium trajectory/control pairs is largely similar to the already established methods in \cite{jacksontangpi,jackson2025quantitativeconvergencedisplacementmonotone}. In addition to the quantitative convergence result for equilibrium trajectory/control pairs, relying on the maximum principle, we obtain a quantitative convergence result also for {\it value functions} of $N$-player games to MFGs. An interesting feature is that the convergence of the value functions is more precise than that of equilibrium trajectory/control pairs. In particular, for the latter the convergence result is in an average sense, i.e. the equilibrium trajectory/control pairs averaged over all $N$ players is shown to converge to the MF equilibrium trajectory/control pair as $N\rightarrow\i$, while for the former we can prove this convergence result in an individual sense.

\medskip

We note that although in the deterministic case distributed equilibria coincide with open-loop equilibria and therefore the convergence result in this case essentially coincides with that of \cite[Theorem 1.6]{jackson2025quantitativeconvergencedisplacementmonotone}, the existence of such equilibria in our work is novel, as the existence of MF equilibria for MFGCs and the existence of open-loop equilibria for corresponding $N$-player games in \cite{jackson2025quantitativeconvergencedisplacementmonotone} rely on the setting of nondegenerate idiosyncratic noise. To the best of our knowledge, the existence of open-loop equilibria for deterministic $N$-player games when agents interact through the distributions of their controls has not been established previously in the literature. 

\medskip

Interestingly, we can prove further convergence results in the case of deterministic  problems. We note that in the case of stochastic problems, due to the lack of comparison between the functions $D_pH$ and $D_p\Ht$, one is in general unable to obtain quantitative convergence results for the gradients of the value functions. This becomes tractable in the deterministic setting,  but it turns out to be far from trivial. We build an \say{intermediate system} which is some sort of interpolation between the MFGC PDE systems arising from stochastic and deterministic problems, respectively. We then derive bounds that are sustained in the vanishing viscosity limit and deduce the desired gradient convergence results. These estimates are also in an individual sense, leading to an immediate consequence: an individual convergence result on equilibrium trajectory/control pairs, instead of in an averaged sense.

\medskip

Lastly, we explicitly study some model problems in the linear quadratic (LQ) case, to showcase the sharpness of some of our standing assumptions. In particular, we construct a LQ--MFGC for which there are infinitely many solutions (hence MF equilibria) as soon as the displacement semimonotonicity condition is violated. Moreover, these examples show that removing a contractivity condition that we impose, would lead to an LQ problem with no quadratic solutions.

\medskip

\textit{Organisation of the rest of the paper.} The short Section \ref{sec:connection} shows the connection between the distributed equilibria and the corresponding PDE systems. Then, in Section \ref{section existence stochastic}, we prove the existence of solutions to the PDE system \eqref{DNash} in the uniformly parabolic setting, hence the existence of distributed equilibria to the stochastic $N$-player game. In Section \ref{section existence deterministic}, we prove the existence of solutions to the first order PDE system \eqref{DNash deterministic}, hence the existence of distributed equilibria in the absence of noise. In Section \ref{section uniqueness}, we prove stability and uniqueness of solutions to the PDE systems \eqref{DNash} and \eqref{DNash deterministic} under displacement semimonotonicity, which is also introduced in this section. Consequently, we present the last piece for wellposedness of distributed equilibria. Section \ref{section MFG} is dedicated to the MFG and to the convergence results. After recalling the setting and the well-posedness theory for the MFG, in Subsection \ref{section convergence} we quantify the difference between distributed equilibria of $N$-player games and MF equilibria of MFGs, showing that the difference vanishes as $N\rightarrow\i$, both at the level of equilibrium trajectory/control pairs and value functions. We end this section with a concentration inequality. In Section \ref{sec:vanishing} we present a vanishing viscosity argument and  the quantitative convergence of gradients of solutions to the PDE system \eqref{DNash deterministic} associated to the deterministic games. In Section \ref{LQ section}, we derive explicitly computable examples in the linear quadratic setting, both for the $N$-player games and the MFGs. Finally, for completeness and for the reader's convenience we end the paper with three short Appendix sections, where we recall some results (known or expected for experts) from convex analysis and nonlinear nonlocal continuity equations.

\medskip

\subsection*{Acknowledgement.}
The first author was supported by the Engineering and Physical Sciences Research Council [Grant Number EP/W524426/1]. The second author has been
supported by the EPSRC New Investigator Award ``Mean Field Games and Master equations'' under award no. EP/X020320/1.

\section{Connection Between Distributed Equilibria and PDE Systems}\label{sec:connection}

\subsection{Verification Theorem}

\begin{assumption}\ \label{Assumption Convexity}
    For all $i=1,\ldots,N$, assume the following holds.\begin{enumerate} 
    
        \item $\Rd\times\Rd\times(\Rd)^{N-1}\times(\Rd)^{N-1}\ni(x,a,y^{-i},A^{-i})\mapsto L^i(x,a,y^{-i},A^{-i})\in C^1$ are jointly convex functions of the variables $(x,a)$ for all $(y^{-i},A^{-i})$, strictly convex in the variable $a$ uniformly over the variables $(x,y^{-i},A^{-i})$.
        
        \item $\textnormal{There exists } C>0,\ \textnormal{such that}\ L^i(x,a,y^{-i},A^{-i})\geq\th(|a|)-C(1+|x|+\sum_{j\neq i}\left(|y^j|+|A^j|\right)$ where $\th:[0,\infty)\rightarrow[0,\infty)$ is superlinear at infinity.
        \end{enumerate}
        \end{assumption}
We will use the following lemma repeatedly throughout the paper.

\begin{lemma}\label{Dudley lemma}
Let $f:(\Rd)^N\to\R$ be a jointly measurable real function, $E^i_t$ be measurable with respect to $\sF^i_t$ such that $f(E^i_t,X^{-i}_t)\in L^1(\Om),$ where $X^{-i}_t$ are as in \eqref{PMP}, then
    \[\int_{(\Rd)^{N-1}} f(E^i_t,x^{-i})\ dm_t^{-i}(x^{-i})=\E\Bigl[f(E^i_t,X^{-i}_t) | \sF^i_t\Bigr]\]
\end{lemma}

\begin{proof}
    $X^j_t$ is independent of $\sF^i_t$ for $j\neq i$, use \cite[Section 10.1, Problem 9]{Dudley_2002}. 
\end{proof}

\begin{theorem}\label{Verification thm}
Suppose Assumption \ref{Assumption Convexity} holds, let $(w,m,\a)$ be a solution to \eqref{DNash}, then $\a\in(\A^{dist})^N$ defines a set of distributed equilibrium strategies.
\end{theorem}

\begin{proof}

Let $u\in\A^{dist}$ be any alternative strategy the $i$-th player deviates to, denote the associated trajectory as $\Xt^i$. We denote as $X^j$ the trajectories of the $j$-th player, $j\neq i$, associated to $\a^j$.

\medskip

Applying It\^o's formula on $w^i(t,\Xt^i_t)$ and using the HJB equation gives
\begin{align*}
    \E\Bigl[w^i(0,\xi^i)\Bigr]&=\E\left[\gt^i(\Xt^i_T,m^{-i}_T)-\int^T_0 \left(\Ht^i(\Xt^i_t,D_xw^i(t,\Xt^i_t),m^{-i}_t,\a^{-i}_t)+D_xw^i(t,\Xt^i_t)\cdot u(t,\Xt^i_t)\right)\ dt\right]\\
    &\leq \E\left[\gt^i(X^i_T,m^{-i}_T)+\int^T_0 \Lt^i(\Xt^i_t,u(t,\Xt^i_t),m^{-i}_t,\a^{-i}_t) dt\right]
\end{align*}
with equality if and only if $u=\a^i$. Indeed, from the consistency relation in \eqref{DNash} we have
\[\a^i(t,\Xt^i_t)=-D_p\Ht^i(\Xt^i_t,D_xw^i(t,\Xt^i_t),m^{-i}_t,\a^{-i}_t),\]
as $-D_p\Ht^i$ maximises the Hamiltonian $\Ht^i$ uniquely (c.f. \cite[(A.25)]{CanSin} and Remark \ref{coercivity remark}), we have for all $a\in\Rd$
\[\Ht^i(\Xt^i_t,D_xw^i(t,\Xt^i_t),m^{-i}_t,\a^{-i}_t)\geq -D_xw^i(t,\Xt^i_t)\cdot a -\Lt^i(\Xt^i_t,a,m^{-i}_t,\a^{-i}_t) \]
with equality if and only if $a=-D_p\Ht^i(\Xt^i_t,D_xw^i(t,\Xt^i_t),m^{-i}_t,\a^{-i}_t)$.

\medskip

By Lemma \ref{Dudley lemma} and the tower law of conditional expectations, we further deduce that 
\begin{align*}
\E\Bigl[w^i(0,\xi^i)\Bigr]&\leq\E\left[g^i(X^i_T,X^{-i}_T)+\int^T_0 L^i(X^i_t,u(t,\Xt^i_t),X^{-i}_t,\a^{-i}(t,X^{-i}_t)) dt\right]\\
&=J^i(u;\a^{-i}).
\end{align*}
with equality if and only if $u=\a^i$, i.e.
\[J^i(\a)\leq J^i(u;\a^{-i}).\]\end{proof}
\begin{remark}\label{rem:4}
    While in the computations in the proof of Theorem \ref{Verification thm} up to the use of Lemma \ref{Dudley lemma} are still valid if we replace $u\in\A^{dist}$ by arbitrary $\sF_t$-adapted processes (i.e. open-loop strategies), we used the fact that  $u(t,X^i_t)$ is in particular $\sF^i_t$-adapted to involve Lemma \ref{Dudley lemma} to go from the lifted running cost $\Lt$ to $L$.
\end{remark}

\subsection{Fully Deterministic Problems}\label{discussion}

\medskip

\begin{theorem}
    Suppose Assumption \ref{Assumption Convexity} holds, let $(w,X,\a)$ be a solution to \eqref{DNash deterministic}, then $\a\in(\A^{dist})^N$ defines a set of distributed equilibrium strategies. Moreover, for all $u\in L^1([0,T])$ 
    \[J^i(\a)\leq J^i(u;\a^{-i})\]
    where the cost functionals are extended in a natural way to \say{open-loop} strategies of the form $u\in L^1([0,T])$. That is, $\a$ induces a set of open-loop equilibrium strategies.
\end{theorem}

\begin{proof}[Sketch of proof]
    This proof follows the same steps as the one of Theorem \ref{Verification thm}. First we note that due to the fully deterministic setting we no longer have expectations nor lifted functions $\gt,\Lt,\Ht$. As noted already in the Remark \ref{rem:4}, the computations in the proof of Theorem \ref{Verification thm} with $u\in\A^{dist}$ replaced by arbitrary measurable functions $u:[0,T]\rightarrow\Rd$ are still valid up to the use of Lemma \ref{Dudley lemma}. Moreover, as we do not need to use the lifted running costs and hence there is no need to invoke Lemma \ref{Dudley lemma}, the computations are in fact valid throughout. 
\end{proof}

One should take care in that in general open-loop equilibria do not necessarily consist of open-loop strategies induced by feedback functions.

\medskip

Open-loop equilibria are also characterised by the Pontrayagin maximum principle. In particular, when $\b=0$ and $\xi^i=z^i$, \eqref{PMP} becomes

\begin{equation}\label{Deterministic PMP}
    \begin{cases}
        X^i_t=z^i-\displaystyle\int^t_0D_pH^i(X^i_s,Y^i_s,X^{-i}_s,\a^{-i}(s,X^{-i}_s))\ ds, &\textnormal{in }[0,T]\\
        Y^i_t=D_xg^i(X^i_T,X^{-i}_T)-\displaystyle \int^T_tD_xH^i(X^i_s,Y^i_s,X^{-i}_s,\a^{-i}(s,X^{-i}_s))\ ds, &\textnormal{in }[0,T].
        \end{cases}
\end{equation}
The system \eqref{Deterministic PMP} is formally connected to a PDE system, which describes the so-called \say{decoupling field} $u:[0,T]\times(\Rd)^N\rightarrow(\Rd)^N$, c.f. ($\textnormal{PS}_{\textnormal{PDE}}$) in \cite{cirant2025nonasymptoticapproachstochasticdifferential},
\begin{equation}\label{decoupling PDE}
    \begin{cases}
        -\p_t u^i+D_xH^i(x^i,u^i,x^{-i},\a^{-i}(t,x^{-i}))+\sum^N_{j=1}D_{x^j}u^iD_pH^j(x^j,u^j,x^{-j},\a^{-j}(t,x^{-j}))=0, &\textnormal{in }(0,T)\times(\Rd)^N\\
        u^i(T,x)=D_xg^i(x^i,x^{-i}), &\textnormal{in }(\Rd)^N.
    \end{cases}
\end{equation}
If $u$ is a classical solution to \eqref{decoupling PDE}, then the solution to \eqref{Deterministic PMP} takes the form $Y^i_t=u^i(t,X_t)$.

\medskip

The decoupling field system \eqref{decoupling PDE} is related to the PDE system \eqref{DNash deterministic} we have considered. In particular, consider \eqref{DNash deterministic} but with initial condition $(t_0,z)\in[0,T]\times (\Rd)^N$, i.e.

\begin{equation}\label{DNash deterministic3}
    \begin{cases}
        -\p_tw^i+H^i(x,D_xw^i,X^{-i}_t,\alb^{-i}(t,X^{-i}_t))=0, &\textnormal{in }(0,T)\times\Rd\\
        \dot{X}^i_t=\alb^i(t,X^i_t), &\textnormal{in }(0,T)\\
    w^i(T,x)=g^i(x,X^{-i}_T),\ X^i_{t_0}=z^i, &\textnormal{in }\Rd
    \end{cases}
\end{equation}
which we assume to be wellposed here, then formally $u^i(t_0,z)=D_xw^i(t_0,z^i)$, where $w$ solves \eqref{DNash deterministic3}. Of course, to make this rigorous one would need suitable a priori estimates for $u.$ We will not consider any such justification in the current paper.

\medskip

This is in the same philosophy as the relation between the usual MFG system and the master equation (c.f. \cite[Section 1.4.3]{cardaliaguetMFG}), the solutions to \eqref{DNash deterministic3} can be considered as characteristics of the quasilinear PDE system \eqref{decoupling PDE}. 

\section{Existence in the Stochastic Case}\label{section existence stochastic}

Consider first the case $\b>0$, we need to make some assumptions on the set of data and the functions  $L^i:\Rd\times\Rd\times(\Rd)^{N-1}\times(\Rd)^{N-1}\rightarrow\R$, $g^i:\Rd\times(\Rd)^{N-1}\rightarrow\R$.
\begin{assumption}\ \label{Assumption L}
    For all $i=1,\ldots,N$, assume the following holds.\begin{enumerate} 
    \item $\m^i\in\prob_2(\Rd)$,  i.e. $\int_\Rd|x|^2\ d\m^i(x)<\i.$
    
        \item $(x,a,y^{-i},A^{-i})\mapsto L^i(x,a,y^{-i},A^{-i})\in C^1$ are (jointly) convex functions of the variables $(x,a)$ for all $(y^{-i},A^{-i})$.
        
        \item $\textnormal{There exists } C>0$ such that $L^i(x,a,y^{-i},0)\geq\th(|a|)-C$ where $\th:[0,\infty)\rightarrow[0,\infty)$ is superlinear at infinity.

        \item $\textnormal{There exists }\l_{i,\max}\geq\l_{i,\min}>0$ such that $\l_{i,\min}I\leq D^2_{aa}L^i\leq\l_{i,\max}I$ uniformly over all variables.

        \item $L^i$ are Lipschitz continuous and semiconcave in the variable $x$ uniformly over all other variables. Moreover $L^i$ have quadratic growth in the variable $y^{-i}$. To be precise, for all $R>0,\ |x|, |a|,|A^{-i}|\leq R,\textnormal{there exists }   C_R>0$ such that $ |L^i(x,a,y^{-i},A^{-i})|\leq C_R(1+\sum_{j\neq i}|y^{j}|^2)$. 

        \item $g^i\in C^1$ are 
        convex, Lipschitz continuous and semiconcave in the variable $x$ uniformly.

        \item $D_aL^i\in C^1$ are Lipschitz continuous in all variables uniformly.

        \item $D_xL^i$ are bounded in the variable $x$ uniformly and have quadratic growth in the variable $y$. To be precise, for all $R>0,\ x\in\Rd, |a|,|A|\leq R,\textnormal{there exists }   C_R>0$ such that $ |D_xL^i(x,a,y^{-i},A^{-i})|\leq C_R(1+\sum_{j\neq i}|y^j|^2)$

        \item $\l^{-1}_{i,\min}\sum_{j\neq i}\norm{D^2_{A^ja}L^i}_{\infty,\op}<1$, where $\l_{i,\min}$ are the same as in 4. 
    \end{enumerate}
\end{assumption}

We denote the set of data as 
\[\textnormal{data}=\{T,d,L,g,\b,\m,p,\l\}.\]
Throughout this section we impose Assumption \ref{Assumption L} without explicitly mentioning this again. Notice that most properties of $L^i$ with respect to the first two variables are directly transferable to $\Lt^i$ by linearity. Examples of running costs $L^i$ satisfying all the assumptions above will be given later (in the symmetric setting) in Example \ref{example}.

\begin{remark}
    If $B:\mathscr{X}\rightarrow \mathscr{B(Y,Z)}$, where $\mathscr{X,Y,Z}$ are Banach spaces, by $\norm{B}_{\infty,\op}$ (as in condition 9 of Assumption \ref{Assumption L}) we mean $\sup_{\mathscr{X}}\norm{B(\cdot)}_{\op}$, where $\norm{\cdot}_\op$ is the operator norm on $\mathscr{B(Y,Z)}$, the space of bounded linear operators from $\mathscr{Y}$ to $\mathscr{Z}$. 
\end{remark}

\begin{remark}
    A key assumption is the contractivity condition 9 in Assumption \ref{Assumption L}, as we will see later this ensures that $D_p\Ht^i$ is a contraction with respect to the last variable. This contraction property has been observed already in \cite[Section 3.1]{GomesDeterministicMFGC} to help with the solvability of the consistency relation arising in the mean field setting. However, as mentioned in the introduction, we will use this contractive property in a different manner as we will not aim to analyse the consistency relation directly.
\end{remark}

\begin{lemma}\label{lemma regularity L}
    $(x,a,m,f)\mapsto D_a\Lt^i(x,a,m,f)$ are uniformly Lipschitz continuous in the variables $(x,a,f)$ and in the variable $m$ whenever $f\in (W^{1,\infty}(\Rd))^{N-1}$, where the topology in the last two variables are taken to be the product topology generated by the 2-Wasserstein topology and the uniform topology, respectively. 
    
    \medskip
    
    In turn, $(x,p,m,f)\mapsto D_p\Ht^i(x,p,m,f)$ are uniformly Lipschitz in the variables $(x,p,f)$ and in $m$ whenever $f\in (W^{1,\i}(\Rd))^{N-1}$. In particular, the uniform Lipschitz constant of $D_p\Ht^i$ in the $f$ variable is given by $\l^{-1}_{i,\min}\norm{D^2_{Aa}L^i}_{\infty,\op}.$ 
\end{lemma}

\begin{proof}
    From condition 7 in Assumption \ref{Assumption L} we see that $D_a\Lt^i(x,a,m,f)$ are uniformly Lipschitz continuous in the variables $(x,a)$. We show the remainder of the statement,
    \begin{equation}\label{D_fa Lt}|D_a \Lt^i (x,a,m^{-i},f^{-i})-D_a\Lt^i(x,a,m^{-i},\overline{f}^{-i})|\leq\sum_{j\neq i}\ \norm{D^2_{A^ja}L^i}_{\infty,\op} \norm{f^j-\overline{f}^j}_\infty\end{equation}
and
    \begin{align*}
        |D_a \Lt^i (x,a,m^{-i},f)-D_a\Lt^i(x,a,\mb^{-i},f)|&\leq CW_1(\otimes_{j\neq i}m^j,\otimes_{j\neq i}\mb^j)\\
        &\leq CW_2(\otimes_{j\neq i}m^j,\otimes_{j\neq i}\mb^j)\\
        &\leq C\left(\sum_{j\neq i}W_2^2(m^j,\mb^j)\right)^{\frac{1}{2}}
        \end{align*}
        where $f\in (W^{1,\i}(\Rd))^{N-1}$, the constant $C=C(\norm{D^2_{y^{-i}a}L^i}_{\infty,\op},\ \Lip(f))$ and the Wasserstein distances $W_r$, $r=1,2$, in the first two lines are taken on the space $\prob_2((\Rd)^{N-1})$. Moreover, the last inequality is due to a well known bound that can be found in \cite[Section 2]{statisticalWassersteinDistance}.

\medskip

 As $D_aL^i\in C^1$, we can deduce that $D_a\Lt^i\in C^1$ (in the sense of Fréchet and intrinsic differentiability for the latter two variables respectively). Furthermore, for all $\phi\in C_b(\Rd)$
 \[\langle D^2_{f^ja}\Lt^i(x,a,m^{-i},f^{-i}),\phi\rangle=\int_{(\Rd)^{N-1}}D^2_{A^ja}L(x,a,x^{-i},f^{-i}(x^{-i}))\phi(x^j)\ dm^{-i}(x^{-i})\]
 ($\langle,\rangle$ is the duality bracket and the integrals on the right are intepreted as Bochner integrals), we have $\norm{D^2_{f^ja}\Lt^i}_{\infty,\op}\leq\norm{D^2_{A^ja}L^i}_{\infty,\op}$.

 \medskip
 
 From the classical identity (\cite[Corollary A.2.7]{CanSin}),

        \begin{equation}\label{first order optimality condition}
        -p-D_a\Lt^i(x,-D_p\Ht^i(x,p,m^{-i},f^{-i}),m^{-i},f^{-i})=0,\end{equation}
we deduce by the (Banach space) implicit function theorem \cite[Theorem 5.9]{lang2012fundamentals} (technically one also needs to lift the measure variable to the space of $L^2$ random variables, see \cite[Section 5.2]{carmonadelarue1}) that $D_p\Ht^i(x,p,m^{-i},f^{-i})\in C^1$. 

\medskip

        Taking the Fréchet derivative  with respect to $f^j$ in \eqref{first order optimality condition}, $j\neq i$, in the Banach space $C_b(\Rd)$ gives
        \[D^2_{aa}\Lt^i(x,-D_p\Ht^i,m^{-i},f^{-i})D^2_{f^jp}\Ht^i(x,p,m^{-i},f^{-i})+D^2_{f^ja}\Lt^i(x,-D_p\Ht^i,m^{-i},f^{-i})=0,\]
        which by (\ref{D_fa Lt}) and condition 4 in Assumption \ref{Assumption L} yields $\norm{D^2_{f^jp}\Ht^i}_{\infty,\op}\leq\l^{-1}_{i,\min}\norm{D^2_{A^ja}L^i}_{\infty,\op}.$ 
        
        \medskip
        
        Differentiating \eqref{first order optimality condition} in the other variables one can prove the uniform Lipschitz continuity of $D_p\Ht^i$ with respect to the other variables.
\end{proof}

\begin{theorem}\label{existence thm stocahstic}
    Let $\b>0$, there exists a solution (according to Definition \ref{def classical solution}) to the system \eqref{DNash}.
\end{theorem}

\begin{proof}
    Let $C_2,C_3,C_9>0$ be constants to be determined, define the convex and closed set $\mathscr{X}$ in the Fréchet space $(C([0,T]\times\Rd;\Rd))^N$ by
\[\mathscr{X}:=\Bigl\{\a\ \Bigl|\ \norm{\a^i}_\infty\leq C_3,\ |\a^i(t,x)-\a^i(s,y)|\leq C_9|t-s|^{\frac{1}{3}}+C_2|x-y|\ \forall (t,x),(s,y)\in[0,T]\times\Rd\ \Bigr\}.\]
The series of constants appearing in the proof $C_n$, $n=0,\ldots,9$ only depend on the set of data, moreover the $C_n$'s do not depend on $C_k$ for any $k>n$.

\medskip

Given $\alb\in X$, define the mapping $S(\alb):=\a$ as follows,
\begin{equation}\label{S(alpha)}
    (S(\alb))^i=\a^i(t,x):=-D_p\Ht^i(x,D_xw^i(t,x),m^{-i}_t,\alb^{-i}_t),
\end{equation}
where $(w,m)$ are the unique solution to the PDE system below
\begin{equation}\label{DSys2}
    \begin{cases}
      \p_t m^i_t-\b\D m_t^i +\na\cdot(m^i_t \alb^i_t)=0, &\textnormal{in }(0,T)\times\Rd,\\
        -\p_tw^i-\b\D w^i+\Ht^i(x,D_xw^i,m^{-i}_t,\alb^{-i}_t)=0, &\textnormal{in }(0,T)\times\Rd,\\
    w^i(T,x)=\gt^i(x,m^{-i}_T),\ m^i_0=\m^i, &\textnormal{in }\Rd.
    \end{cases}
\end{equation}
    Indeed, the KFP equation admits a unique distributional solution $m^i_t=\Leb(X^i_t)\in C([0,T];\prob_2(\Rd))$,
    where $X^i_t$ is the unique solution to the SDE
    \[dX^i_t=\alb^i(t,X^i_t) dt+\sqrt{2\b}dB^i_t,\ \ X^i_0=\xi^i.\]
    Next, the HJB equation admits a unique classical solution (a priori $w^i$ is only a viscosity solution, however by parabolicity and uniqueness we can deduce that it is in fact classical, see below)
    \[w^i(t,x)=\inf\Biggl\{\E\left[\int^T_t \Lt^i(Y_s,u_s,m^{-i}_s,\alb^{-i}_s)\ ds+\gt^i(Y_T,,m^{-i}_T)\right]\Biggl\}\]
    where the infimum is taken over all $\sF^i$-progressively measurable processes and
    \[dY_s=u_sds+\sqrt{2\b}dB^i_s,\ \ Y_t=x.\]
    By classical results in stochastic control theory (see \cite[Chapter 4, Theorem 1]{krylov1980controlled}, \cite[Proposition 4.5]{yong2012stochastic} and Theorem \ref{convexity thm}), under the global Lipschitz continuity, semiconcavity and convexity 
    assumptions on the running cost $\Lt^i( \cdot,\cdot,m^{-i}_t,\alb^{-i}_t)$ and the final cost $\gt^i( \cdot,m^{-i}_T)$ (which are inferred from conditions 3,5,6 in Assumption \ref{Assumption L}), there exists constants $C_0$ and $C_1$ only depending on the set of data such that we have
    \[\norm{D_xw^i}_\i\leq C_0\ \ \textnormal{and}\ \ \norm{D^2_{xx}w^i}_{\i,\op}\leq C_1.\]
    It follows that there exists a constant we choose to be $C_2>0$ such that for all $t\in[0,T], x,y\in\Rd$
    \begin{align*}
        |\a^i(t,x)-\a^i(t,y)|&= |D_p\Ht^i(x,D_xw^i(t,x),m^{-i}_t,\alb^{-i}_t)-D_p\Ht^i(y,D_xw^i(t,y),m^{-i}_t,\alb^{-i}_t)|\\
        &\leq \norm{D^2_{xp}\Ht^i}_{\infty,\op}|x-y|+\norm{D^2_{pp}\Ht^i}_{\infty,\op}|D_xw^i(t,x)-D_xw^i(t,y)|)\\
        &\leq C_2|x-y|.
    \end{align*}
 We infer from condition 3 in Assumption \ref{Assumption L} that $\Lt^i$ satisfies \eqref{coercive}\footnote{As before we technically first need to do a $L^2$ lifting of the measure variable.}, hence by Theorem \ref{legendre} in the appendix, there exists a constant $C_{C_0}>0$ such that
    \[\sup_{|p|\leq C_0}\sup_{(t,x)\in[0,T]\times\Rd}|D_p\Ht^i(x,p,m^{-i}_t,0)|\leq C_{C_0},\]
    it follows that \[\norm{\a^i}_\infty\leq C_{C_0}+\sum_{j\neq i}\norm{D^2_{f^jp}\Ht^i}_{\i,\op}\norm{\alb^j}_\i\leq C_{C_0}+\l_{i,\min}^{-1}\sum_{j\neq i}\norm{D^2_{f^ja}\Lt^i}_{\i,\op}\norm{\alb^j}_\i.\]
We now choose $C_3>0$ to be such that 
\[C_3\geq \frac{C_{C_0}}{1-\l_{i,\min}^{-1}\sum_{j\neq i}\norm{D^2_{f^jp}\Lt^i}_{\i,\op}}.\]
 Since $\alb\in X$, $\norm{\alb^j}_\i\leq C_3$ and thus
\[\norm{\alb}_\i\leq C_{C_0}+C_3\l_{i,\min}^{-1}\sum_{j\neq i}\norm{D^2_{f^jp}\Lt^i}_{\i,\op}\leq C_3 .\]
Moreover, as $-D_x\Ht^i(x,p,m^{-i}_t,\alb^{-i}_t)=D_x \Lt^i(x,-D_x\Ht^i(x,p,m^{-i}_t,\alb^{-i}_t),m^{-i}_t,\alb^{-i}_t)$
\begin{align*}\sup_{|p|\leq C_0}\sup_{(t,x)\in[0,T]\times\Rd}|D_x\Ht^i(x,p,m^{-i}_t,\alb^{-i}_t)|&=\sup_{|a|\leq C_3}\sup_{(t,x)\in[0,T]\times\Rd}|D_x \Lt^i(x,a,m^{-i}_t,\alb^{-i}_t)|\\
&\leq C_{C_3}\left(1+\sum_{j\neq i}\int_\Rd|y|^2\ dm^{j}_t(y)\right).
\end{align*}
By \cite[Theorem 3.2.2]{Zhang2017}, there exists a constant $C_4>0$ such that
\[\int_\Rd|y|^2\ dm^{j}_t(y)=\E\left[|X^j_t|^2\right]\leq C_{C_2}( \E\left[|\xi^j|^2\right]+T( C_3+\sqrt{2\b})^2)\leq C_4,\]
thus there exists $C_5>0$ such that $\sup_{|p|\leq C_0}\sup_{(t,x)\in[0,T]\times\Rd}|D_x\Ht^i(x,p,m^{-i}_t,\alb^{-i}_t)|\leq C_5$.
\medskip

For $j=1,\ldots,d$, by classical parabolic Schauder estimates, $\p_{x_j}w^i\in C^{1,2}$. Indeed, by \cite[Theorem 6.1]{ladyzhenskaia} or  \cite[Theorem 12.14]{lieberman1996second} we have that $w\in C^{1+\d/2,2+\d}_{\textnormal{loc}}$, for some $\d\in(0,1)$, then we can bootstrap up to the desired regularity using \cite[Theorem 8.12.1]{krylov1996lectures} (as $D_xw^i$ is bounded, with a cut off argument, without loss of generality we can take $\Ht^i$ to be Lipschitz in $p$, so that the theorems cited above applies, see \cite[Proof of Theorem 1.5]{cardaliaguetMFG}).

\medskip

Denote the $j$-th component of $D_xw^i$ as $u^{ji}:=\p_{x_j}w^i$, then $u^{ji}$ satisfies the inhomogeneous heat equation
\begin{equation}\label{heat eqn}
    \begin{cases}
        -\p_tu^{ji}-\b\D u^{ji}=f^{ji}
        , &\textnormal{in }(0,T)\times\Rd,\\
        u^{ji}(T,x)=\p_{x_j}\gt^i(x,m^{-i}_T), &\textnormal{in }\Rd,
    \end{cases}
\end{equation}
where $f^{ji}(t,x)=\p_{x_j}\Ht^i(x,D_xw^i,m^{-i}_t,\alb^{-i}_t)+D_p\Ht^i(x,D_xw^i,m^{-i}_t,\alb^{-i}_t)D_xu^{ji}$. With the bounds we have derived, we find that there exists a $C_6>0$ such that $\norm{f^{ji}}_\infty\leq C_6.$

\medskip

We use an idea from Kružkov \cite{Kruzhkov}, which was also used in \cite[Lemma B.1]{CirantRedaelli}, in order to transfer the spatial Lipschitz regularity of $u^{ji}$ (note that $\norm{D_xu^{ji}}_\infty\leq C_1$) to Hölder regularity in time. In particular, there exists a constant $C_7>0$ (depending only on $C_0,C_1,C_6,\b,d,T)$ such that for all $x\in\Rd, t,s\in[0,T]$
\[|u^{ji}(t,x)-u^{ji}(s,x)|\leq C_7|t-s|^\frac{1}{3}.\]
By classical results on KFP equations (e.g. \cite[Chapter 1, Remark 1.6]{cardaliaguetMFG}), we also have 
\begin{equation}\label{continuity of measure flow}W_2(m^j_t,m^j_s)\leq \sqrt{T}C_3|t-s|^\frac{1}{2}\leq T^\frac{2}{3}C_3|t-s|^\frac{1}{3}.\end{equation}
Then ($D_{m^{-i}}$ below denotes the intrinsic derivative, c.f. \cite[Chapter 5]{carmonadelarue1}, simply serving as a placeholder for the Lipschitz constant at the moment)
\begin{align*}
    |\a^i(t,x)-\a^i(s,x)|&=|D_p\Ht^i(x,D_xw^i(t,x),m^{-i}_t,\alb^{-i}_t)-D_p\Ht^i(x,D_xw^i(s,x),m^{-i}_s,\alb^{-i}_s)|\\
    &\leq \norm{D^2_{pp}\Ht^i}_{\infty,\op}|D_xw^i(t,x)-D_xw^i(s,x)|+\norm{D^2_{m^{-i}p}\Ht^i}_{\infty,\op}\left(\sum_{j\neq i}W_2^2(m^j_t,m^j_s)\right)^\frac{1}{2}\\&+\sum_{j\neq i}\ \norm{D^2_{f^jp}\Ht^i}_{\infty,\op}\norm{\alb^j(t,\cdot)-\alb^j(s,\cdot)}_\infty.
\end{align*}

By  previous bounds and Lemma \ref{lemma regularity L} we find a constant $C_8>0$ such that 
\[|\a^i(t,x)-\a^i(s,x)|\leq\left(C_8|t-s|^\frac{1}{3}+\l_{i,\min}^{-1}\sum_{j\neq i}\ \norm{D^2_{A^ja}L^i}_{\infty,\op}\norm{\alb^j(t,\cdot)-\alb^j(s,\cdot)}_\infty\right).\]
Recalling condition 9 of Assumption \ref{Assumption L}, we now choose $C_9>0$ to be such that 
\[C_9\geq\frac{C_8}{1-\l_{i,\min}^{-1}\sum_{i\neq j}\norm{D^2_{A^ja}L^i}_{\infty,\op}},\]
since $\alb\in \mathscr{X}$, $\norm{\alb^j(t,\cdot)-\alb^j(s,\cdot)}_\infty\leq C_9|t-s|^\frac{1}{3},$ then
\begin{align*}|\a^i(t,x)-\a^i(s,x)|&\leq\left(C_8+\l_{i,\min}^{-1}\sum_{j\neq i}\ \norm{D^2_{A^ja}L^i}_{\infty,\op}C_9\right)|t-s|^\frac{1}{3}\\
&\leq C_9|t-s|^\frac{1}{3}\end{align*}
and thus $S(\alb)=\a\in \mathscr{X}$.

\medskip

We have shown that $S:\mathscr{X}\rightarrow \mathscr{X}$ is well defined. Moreover, by Arzelà--Ascoli theorem \cite[Theorem 47.1]{munkres2000topology}, $\mathscr{X}$ is compact in $(C([0,T]\times\Rd;\Rd))^N$ equipped with the (product) topology of local uniform convergence. It remains to show that $S$ is continuous in order to apply the Schauder--Tychonoff fixed point theorem \cite[Theorem 3.2]{bonsall1962lectures}. 

\medskip

To this end let $\alb_n,\alb\in \mathscr{X}$, $n\in\N$ be such that $\alb_n\rightarrow \alb$ locally uniformly on $[0,T]\times\Rd$, as $n\to+\infty$. By classical stability of Fokker--Planck equations (inferred from stability for SDEs \cite[Theorem 3.4.2]{Zhang2017}) we have $m_n\rightarrow m$ in $(C([0,T];\prob_1))^N$, as $n\to+\infty$, where $m_n,m$ denotes the solution to the Fokker--Planck equations associated to the vector fields $\alb_n,\alb$ respectively.

\medskip

Denote by $w_n,w$ the corresponding solution to the Hamilton-Jacobi equations associated to $(m_n,\alb_n),(m,\alb)$. By the previously derived bounds on the solution to the Hamilton-Jacobi equations, which are independent of $n\in\N$, together with the bound
\begin{align*}
    |\p_t w^i_n(t,x)|&\leq\b|\D w^i_n(t,x)|+|\Ht^i(x,D_xw^i_n(t,x),m^{-i}_n(t),\alb^{-i}_n(t))|\\
    &\leq \b C_1+C_3C_0+C_{|x|+C_3}(1+(N-1)C_4),
\end{align*}
where we have used the identity (and growth conditions on $L^i$)
\[\Ht^i(x,p,m,f)=p\cdot D_p\Ht^i(x,p,m,f)-\Lt^i(x,-D_p\Ht^i(x,p,m,f),m,f),\]
we can extract from every subsequence $(w_{n_j})_{n_j}$, a further subsequence which we do not relabel such that $w_{n_j}\rightarrow \wt, D_xw_{n_j}\rightarrow D_x\wt$ locally uniformly, as $j\to+\infty$. By stability (e.g. from Duhamel's formula) and uniqueness of the HJB equation we have that $\wt=w$. Therefore we must have that the whole sequence converges, in particular $D_xw_n\rightarrow D_x w$ locally uniformly, as $n\to+\infty$. Then we can pass to the limit in the expression for $S(\alb_n)$ and conclude that $S$ is in fact a continuous mapping. 

\medskip

Lastly, we apply the Schauder fixed point theorem to conclude the existence of a fixed point $\a\in \mathscr{X}$. Clearly, $\a$ generates a triplet $(w,m,\a)$ which is a solution of (\ref{DNash}).
    %
\end{proof}

\section{Existence in the Deterministic Case}\label{section existence deterministic}

Consider next the case of zero idiosyncratic noise dynamics with random initial conditions, i.e. we investigate the existence of solutions of (\ref{DNash}) with $\b=0.$

\medskip

In this case, we can relax some assumptions, in particular conditions 3, 5, 6 are relaxed from Assumption \ref{Assumption L} and the boundedness condition on $D_xL$ is removed.

\begin{assumption}\label{Assumption L loc}\
    \begin{enumerate} 
    \item $\m^i\in\prob_2(\Rd)$.
    
    \item $(x,a,y^{-i},A^{-i})\mapsto L^i(x,a,y^{-i},A^{-i})\in C^1$ are convex functions of the variables $(x,a)$ for all $(y^{-i},A^{-i})$.
        
        \item $\textnormal{There exists } C>0,\ \textnormal{such that}\ L^i(x,a,y^{-i},A^{-i})\geq\th(|a|)-C(1+|x|)$ where $\th:[0,\infty)\rightarrow[0,\infty)$ is superlinear.

        \item $\textnormal{There exists }\l_{i,\max}\geq\l_{i,\min}>0$ such that $\l_{i,\min}I\leq D^2_{aa}L^i\leq\l_{i,\max}I$ uniformly over all variables.

        \item $L^i$ are locally Lipschitz continuous and (globally) semiconcave in the variable $x$ uniformly over all other variables. Moreover $L^i$ have quadratic growth in the $(y^{-i},A^{-i})$. To be precise, for all $R>0,\ |x|, |a|\leq R,\textnormal{there exists }   C_R>0$ such that $ |L^i(x,a,y^{-i},A^{-i})|\leq C_R(1+\sum_{j\neq i}(|y^{j}|^2+|A^j|^2))$

        \item $g^i\in C^1$ convex, locally Lipschitz continuous and (globally) semiconcave in the variable $x$ uniformly.

 \item $D_xL^i$ have quadratic growth in the variables $(y^{-i},A^{-i})$. To be precise, for all $R>0,|x|,|a|\leq R,\textnormal{there exists }   C_R>0$ such that $ |D_xL^i(x,a,y^{-i},A^{-i})|\leq C_R(1+\sum_{j\neq i}(|y^{j}|^2+|A^j|^2))$

        \item $D_aL^i\in C^1$ are (globally) Lipschitz continuous in all variables uniformly.
        

        \item $\l^{-1}_{i,\min}\sum_{i\neq j}\norm{D^2_{A^ja}L^i}_{\infty,\op}<1$ 
    \end{enumerate}
\end{assumption}

\medskip

\begin{theorem}\label{existence thm intermediate}
    Let Assumption \ref{Assumption L loc} holds, $\b=0$. Furthermore, suppose $\m^i$ are compactly supported, then there exists a solution to the system \eqref{DNash}.
\end{theorem} 

\begin{proof}
    The proof is essentially the same as the proof of Theroem \ref{existence thm stocahstic}, the difference mostly being that we now obtain the equicontinuity in time from (\ref{heat eqn}) directly. Define the set 
\[\mathscr{X}:=\Bigl\{\a\ |\ |\a^i(t,x)|\leq C_3,\ |\a^i(t,x)-\a^i(s,y)|\leq C_{9}|t-s|+C_2|x-y|\ \forall\ (t,x),(s,y)\in [0,T]\times B_R\ \Bigr\}\]
where $C_n=C_n(R):[0,\infty)\rightarrow[0,\infty)$ are increasing functions depending only on the set of data, 
we will refer to such functions as numerical functions. In particular $C_2$ is just a constant, so that the vector fields in $\mathscr{X}$ are globally Lipschitz continuous in space. 

\medskip
 
 Define the mapping $S(\alb)=\a$ as in \eqref{S(alpha)}, there exist numerical function $C_0$ and constant $C_1$ such that
\[|D_xw^i|\leq C_0(R)\ \ \textnormal{on}\ [0,T]\times B_R\ \ \textnormal{and}\ \ \norm{D^2_{xx}w^i}_{\infty,\op}\leq C_1,\]
where $(w,m)$ are the unique solution to \eqref{DSys2} with $\b=0.$ These estimates follow from the same reasoning as in the proof of Theorem \ref{existence thm stocahstic}.

\medskip
    It follows that there exists constant $C_2>0$ such that for all $t\in[0,T], x,y\in \Rd$
    \begin{align*}
        |\a^i(t,x)-\a^i(t,y)|\leq C_2|x-y|.
    \end{align*}
    By Theorem \ref{Legendre thm 2}, we can choose $C_3$ to be such that for all $(t,x)\in[0,T]\times B_R$,
    \[|\a^i(t,x)|=|D_p\Ht^i(x,D_xw^i(t,x),m^{-i}_t,\alb^{-i}_t)|\leq C_3(R).\]
    %
  %
  %
  We need an estimate on the support of $m_t^i=(\varphi^i(t;\cdot))_\sharp\m^i$, where we denoted $\varphi$ as the flow map, i.e. for $y\in\spt(\m^i)$
\[\begin{cases}
    \dot{\varphi}^i(t;y)=\a^i(t,\varphi^i(t;y)), &\textnormal{in }(0,T),\\
    \varphi^i(0;y)=y.
\end{cases}\]
  
  It follows from the assumption that $\m^i$ is compactly supported and the estimate
\[|\varphi^i(t;y)|\leq e^{T(C_2+C_3(0))}(|y|+(C_2+C_3(0))T),\]
that there is a constant $C_4>0$ such that the support of $m^i_t$ is contained in the ball $B_{C_4}$ for all $t\in[0,T].$ 

\medskip

    Next, using the growth bound condition 7 in Assumption \ref{Assumption L loc},
    \begin{align*}\sup_{|p|\leq C_0(R)}\sup_{(t,x)\in[0,T]\times\Rd}|D_x\Ht^i(x,p,m^{-i}_t,\alb^{-i}_t)|&=\sup_{|a|\leq C_3(R)}\sup_{(t,x)\in[0,T]\times\Rd}|D_x \Lt^i(x,a,m^{-i}_t,\alb^{-i}_t)|\\
    &\leq C_{C_3}(R)\left(1+\sum_{j\neq i}\left(\int_{B_{C_4}}|y^{j}|^2+|\alb^j(t,y^j)|^2\ dm^{j}_t\right)\right)\\
    &\leq C_{C_3}(R)+(N-1)(C_4^2+C_3(C_4))=:C_5(R).
\end{align*}
 We now understand (\ref{heat eqn}) (with $\b=0$) in the a.e. sense, which is allowed by the $C^{1,1}$ estimates and the fact that we deduce from control theory that $w^i$, $i\in\{1,\dots,N\},$ are also Lipschitz continuous in time (so the Hamilton--Jacobi equation to \eqref{DSys2} is in fact satisfied pointwise a.e.). 
    \begin{equation}\label{time estimate Dxw}|\p_tD_xw^i(t,x)|\leq C_4(R)+C_1 C_3(R)\leq C_6(R)\ \ \ae\ (t,x)\in[0,T]\times B_R.\end{equation}
%
For some constant $C_7>0$, (\ref{continuity of measure flow}) now becomes
    \[W_2(m^j_t,m^j_s)\leq\E\Bigl[|X^j_t-X^j_s|^2\Bigr]^\frac{1}{2}\leq\E\left[\left(\int^t_s |\alb^j(\t,X^j_\t)|\ d\t\right)^2\right]^\frac{1}{2}\leq C_7|t-s|.\]
Then it is a straightforward calculation that for all $t,s\in[0,T],x\in B_R$
\begin{align*}&\ |D_a \Lt^i (x,D_xw^i(t,x),m^{-i}_t,\alb_t^{-i})-D_a\Lt^i(x,D_xw^i(s,x),m^{-i}_s,\alb_s^{-i})|
\\ \leq&\left(C_7\norm{D^2_{m^{-i}a}\Lt^i}_{\infty,\op}+\norm{D^2_{pa}\Lt^i}_{\infty,\op}C_5(R)+\sum_{j\neq i}\ \norm{D^2_{A^ja}L^i}_{\infty,\op}\int_{\spt(m^j_0)} C_{9}(|y|)\ dm^j_0(y)\right)|t-s|\\
\leq&\left(C_7\norm{D^2_{m^{-i}a}\Lt^i}_{\infty,\op}+\norm{D^2_{pa}\Lt^i}_{\infty,\op}C_5(R)+\sum_{j\neq i}\ \norm{D^2_{A^ja}L^i}_{\infty,\op}C_{9}(C_6)\right)|t-s|.\end{align*}
    %
    Transferring the Lipschitz regularity in time from $D_a\Lt$ to $D_p\Ht$ with a similar method, as in Lemma \ref{lemma regularity L}, we find a numerical function $C_8$ such that
    \[|\a^i(t,x)-\a^i(s,x)|\leq\left( C_{8}(R)+\l_{i,\min}^{-1}\sum_{j\neq i}\ \norm{D^2_{A^ja}L^i}_{\infty,\op}C_9(C_6)\right)|t-s|.\]
    Without loss of generality, we can redefine $C_8(R):=C_8(C_6)$ for $R\in[0,C_6]$, now choose $C_9(R)$ such that
    \[C_9(R)\geq\frac{C_8(R)}{1-\l_{i,\min}^{-1}\sum_{i\neq j}\norm{D^2_{A^ja}L^i}_{\infty,\op}},\]

    and 
    proceed as in the proof of Theorem \ref{existence thm stocahstic} to deduce the existence of a fixed point $\a$. Notice that a fixed point $\a$ indeed generates $w\in C^1([0,T]\times\Rd)$, according to \eqref{time estimate Dxw} $D_xw^i$ is further continuous in time, then from the Hamilton--Jacobi equation we see that $\p_tw^i$ must coincide with a continuous function a.e., i.e. continuous up to redefining. 
\end{proof}

Note that the existence of solutions to \eqref{DNash deterministic} follows as a special case of the theorem above. 

\begin{remark}
    We expect that it is possible to remove the assumption that initial measures are compactly supported, if we keep careful track of the polynomial growth of the constants $C_n$ and use moment bounds of the measure flow $m_t$. However, for the simplicity of the presentation of the results, we chose not to do this here.
\end{remark}

\begin{remark}
   Although we have initially defined the lifted functions $\Lt,\Ht$ on $(C_b(\Rd))^{N-1}$ for the last variable, as $\a$ is no longer expected to be bounded in the deterministic case, we need to relax this to the Fréchet space $(C(\Rd))^{N-1}$. It is no longer possible to deduce that the functions $D_p\Ht^i$ are Lipschitz continuous in the last variable, as Lemma \ref{lemma regularity L} relies on the implicit function theorem which no longer holds on  Fréchet spaces. Therefore, it is necessary to instead treat the functions $\Lt,\Ht$ as a function of time instead, as we have done in the proof above.
\end{remark}

\section{Uniqueness under Semimonotonicity}\label{section uniqueness}
We shall assume the following displacement semimonotonicity assumption is satisfied. Note that this condition is a natural generalisation of the notion of displacement semimonotonicity in the existing literature to the case of nonsymmetric cost functions with interactions through controls, see \cite[Assumption 2.3]{cirant2025nonasymptoticapproachstochasticdifferential} and \cite[Assumption 2.3, Lemma 3.4]{jackson2025quantitativeconvergencedisplacementmonotone}.

\medskip

From now on, generic constants $C,C_\e,C_\d$ may increase from line to line while retaining the same notation.
\begin{assumption}\label{Assumption semimon}
There exists constants $C_{L,a}>0,\ C_{L,x},C_g\geq0$ such that
\begin{equation}\label{C disp}
    C_{disp}:=C_{L,a}-TC_g-\frac{T^2}{2}C_{L,x}>0,
\end{equation}
 for all $(x,a),(\xb,\ab)\in(\Rd)^N\times(\Rd)^N$
    \begin{align*}
        \sum^N_{i=1}\Biggl( \Bigl(D_aL^i&(x^i,a^i,x^{-i},a^{-i})-D_aL^i(\xb^i,\ab^i,\xb^{-i},\ab^{-i})\Bigr)\cdot(a^i-\ab^i)\\
        &+ \Bigl(D_xL^i(x^i,a^i,x^{-i},a^{-i})-D_xL^i(\xb^i,\ab^i,\xb^{-i},\ab^{-i})\Bigr)\cdot(x^i-\xb^i)\Biggr)\\
         &\ \ \ \ \ \ \ \ \ \ \ \ \geq C_{L,a}|a-\ab|^2-C_{L,x}|x-\xb|^2
    \end{align*}
and
\[\sum^N_{i=1}\Bigl(D_xg^i(x^i,x^{-i})-D_xg^i(\xb^i,\xb^{-i})\Bigr)\cdot(x^i-\xb^i)\Biggr)\geq-C_g|x-\xb|^2\]
\end{assumption}

 Note that if $C_{L,x},C_g=0$, we say that a displacement monotonicity condition is satisfied instead. 

\begin{assumption}\label{assumption DxL}
    $D_aL^i,D_xL^i,D_xg^i$ are Lipschitz continuous with respect to all variables uniformly. Moreover, suppose conditions 1, 2, 3, 4 in Assumption \ref{Assumption L} hold.
\end{assumption}

The proof of the following stability result is based on the method in \cite[Proposition 4.1]{jackson2025quantitativeconvergencedisplacementmonotone}.

\begin{theorem}\label{stability theorem}
    Under Assumption \ref{Assumption semimon}, \ref{assumption DxL}, let $\b\geq0,$ $(w,m,\a),(\overline{w},\mb,\alb)$ be two solutions to \eqref{DNash} with initial measures $\m,\Bar{\m}$ respectively. Then there exists a constant $C>0$ such that
    \[\sup_{t\in[0,T]}\sum_{i=1}^NW_2(m^i_t,\mb^i_t)\leq C\sum^N_{i=1}W_2(\m^i,\overline{\m}^i)\]
\end{theorem}

\begin{proof}
    We work with the PMP system \eqref{PMP}. Let us define $A^i_t:=\a^i(t,X^i_t)$ and note that 
    $$A^i_t=-D_p\Ht^i(X^i_t,Y^i_t,m^{-i}_t,\alb^{-i}_t).$$ 
    Next, we define $\Yb^i_t:=D_x\overline{w}^i(t,\Xb^i_t),\ \Ab^i_t:=\alb(t,\Xb^i_t)$, where $\Xb$ solves the forward SDE in \eqref{PMP} associated to $(\overline{w},\mb,\alb)$. 

    \medskip
    As $D_a\Lt^i(x,-D_p\Ht^i(x,p,m^{-i}_t,\alb^{-i}_t),m^{-i}_t,\a^{-i}_t)=-p,$
    \begin{align*}
       \E[(A^i_t-\Ab^i_t)\cdot(Y^i_t-\Yb^i_t)]&=-\E\Bigl[(A^i_t-\Ab^i_t)\cdot\bigl(D_a\Lt^i(X^i_t,A^i_t,m^{-i}_t,\a^{-i}_t)-D_a\Lt^i(\Xb^i_t,\Ab^i_t,\mb^{-i}_t,\alb^{-i}_t)\bigr)\Bigr]\\
       &=-\E\Bigl[\E\Bigl[(A^i_t-\Ab^i_t)\cdot\bigl(D_aL^i(X^i_t,A^i_t,X^{-i}_t,A^{-i}_t)-D_aL^i(\Xb^i_t,\Ab^i_t,\Xb^{-i}_t,\Ab^{-i}_t)\bigr)\Bigr|\sF^i_t\Bigr]\Bigr]\\
       &=-\E\Bigl[(A^i_t-\Ab^i_t)\cdot\bigl(D_aL^i(X^i_t,A^i_t,X^{-i}_t,A^{-i}_t)-D_aL^i(\Xb^i_t,\Ab^i_t,\Xb^{-i}_t,\Ab^{-i}_t)\bigr)\Bigr]
    \end{align*}
    where we used Lemma \ref{Dudley lemma} and the tower law of conditional expectation.

    \medskip

    Similarly, as $D_x\Ht^i(x,p,m^{-i}_t,\a^{-i}_t)=-D_x\Lt^i(x,-D_p\Ht^i(x,p,m^{-i}_t,\a^{-i}_t),m^{-i}_t,\a^{-i}_t),$ we have
    \begin{align*}
        &\E\left[\bigl(D_x\Ht^i(X^i_t,Y^i_t,m^{-i}_t,\a^{-i}_t)-D_x\Ht^i(\Xb^i_t,\Yb^i_t,\mb^{-i}_t,\alb^{-i}_t)\bigr)\cdot(X^i_t-\Xb^i_t)\right]\\
        =-&\E\Bigl[\E\Bigl[\bigl(D_xL^i(X^i_t,A^i_t,X^{-i}_t,A^{-i}_t)-D_xL^i(\Xb^i_t,\Ab^i_t,\Xb^{-i}_t,\Ab^{-i}_t)\bigr)\cdot(X^i_t-\Xb^i_t)\Bigr|\sF^i_t\Bigr]\Bigr]\\
        =-&\E\Bigl[\bigl(D_xL^i(X^i_t,A^i_t,X^{-i}_t,A^{-i}_t)-D_xL^i(\Xb^i_t,\Ab^i_t,\Xb^{-i}_t,\Ab^{-i}_t)\bigr)\cdot(X^i_t-\Xb^i_t)\Bigr].
    \end{align*}
    Lastly,
    \[\E\left[\bigl(D_x\gt^i(X^i_T,m^{-i}_T)-D_x\gt^i(\Xb^i_T,\mb^{-i}_T)\bigr)\cdot(X^i_T-\Xb^i_T)\right]=\E\Bigl[\bigl(D_xg^i(X^i_T,X^{-i}_T)-D_xg^i(\Xb^i_T,\Xb^{-i}_T)\bigr)\cdot(X^i_T-\Xb^i_T)\Bigr].\]
    By It\^o's formula, the idendities derived above and Assumption \ref{Assumption semimon}
    \begin{align*}
        &\sum^N_{i=1}\E\Bigl[(Y^i_0-\Yb^i_0)\cdot(X^i_0-\Xb^i_0)\Bigr]\\
        =&\sum^N_{i=1}\E\Biggl[\bigl(D_x\gt^i(X^i_T,m^{-i}_T)-D_x\gt^i(\Xb^i_T,\mb^{-i}_T)\bigr)\cdot(X^i_T-\Xb^i_T)-\int^T_0 (A^i_t-\Ab^i_t)\cdot(Y^i_t-\Yb^i_t)\\
        &\ \ \ \ \ \ \ \ \ \ \ \ -\bigl(D_x\Ht^i(X^i_t,Y^i_t,m^{-i}_t,\a^{-i}_t)-D_x\Ht^i(\Xb^i_t,\Yb^i_t,\mb^{-i}_t,\alb^{-i}_t)\bigr)\cdot(X^i_t-\Xb^i_t)\ dt\Biggr]\\
        \geq&\E\left[-C_g|X_T-\Xb|^2+\int^T_0C_{L,a}|A_t-\Ab_t|^2-C_{L,x}|X_t-\Xb_t|^2\ dt\right]. 
    \end{align*}
    Following the arguments of \cite[Proposition 4.1]{jackson2025quantitativeconvergencedisplacementmonotone} (see also similar arguments in the proof of Theorem \ref{conv thm 1}), for all $\e>0$, there is $C_\e>0$ such that 
    \begin{equation}\label{drift estimate}
    \E\left[\int^T_0|A_t-\Ab_t|^2\ dt\right]\leq\E\Bigl[\e|Y_0-\Yb_0|^2+C_\e|X_0-\Xb_0|^2\Bigr].\end{equation}
    Next, by a similar argument as \cite[Theorem 4.2.3]{Zhang2017}
    \begin{align*}
        &\E\Bigl[|Y_0-\Yb_0|^2\Bigr]\\
        \leq C&\sum^N_{i=1}\E\Biggl[|D_x\gt^i(X^i_t,m^{-i}_t)-D_x\gt^i(\Xb^i_t,\mb^{-i}_t)|^2
        \\&+\left(\int^T_0 |D_x\Ht^i(X^i_t,Y^i_t,m^{-i}_t,\a^{-i}_t)-D_x\Ht^i(\Xb^i_t,\Yb^i_t,\mb^{-i}_t,\alb^{-i}_t)|\ dt\right)^2\Biggr]\\
        \leq C&\E\Biggl[|X_T-\Xb_T|^2+T\sum^N_{i=1}\int^T_0 |D_xL^i(X^i_t,A^i_t,X^{-i}_t,A^{-i}_t)-D_xL^i(\Xb^i_t,\Yb^i_t,\Xb^{-i}_t,\Ab^{-i}_t)|^2\ dt\Biggr]
\\ \leq C&\E\left[\sup_{t\in[0,T]}|X_t-\Xb_t|^2+\int^T_0|A_t-\Ab_t|^2\ dt\right]\\
\leq C&\e\E\Bigl[|Y_0-\Yb_0|^2\Bigr]+C_\e\E\Bigl[|X_0-\Xb_0|^2\Bigr]+C\E\left[\sup_{t\in[0,T]}|X_t-\Xb_t|^2\right]
    \end{align*}
    where we used Lemma \ref{Dudley lemma}, conditional Jensen's inequality and (\ref{drift estimate}).

    \medskip

    For $\e>0$ small enough, we arrive at the estimate
    \begin{equation}\label{Y estimate}
        \E\Bigl[|Y_0-\Yb_0|^2\Bigr]\leq\frac{1}{1-C\e}\left(C_\e\E\Bigl[|X_0-\Xb_0|^2\Bigr]+C\E\left[\sup_{t\in[0,T]}|X_t-\Xb_t|^2\right]\right).
    \end{equation}
    Lastly, arguing analogously as \cite[Theorem 3.2.2]{Zhang2017}, we have
    \begin{align*}
        \E\left[\sup_{t\in[0,T]}|X_t-\Xb_t|^2\right]&\leq C\E\left[|X_0-\Xb_0|^2+T\int^T_0 |A_t-\Ab_t|^2\ dt\right]
        \\&\leq C_\e\E\Bigl[|X_0-\Xb_0|^2\Bigr]+\e \E\Bigl[|Y_0-\Yb_0|^2\Bigr]\\
        &\leq C_\e\E\Bigl[|X_0-\Xb_0|^2\Bigr]+\frac{\e}{1-C\e}\left(C_\e\E\Bigl[|X_0-\Xb_0|^2\Bigr]+C\E\left[\sup_{t\in[0,T]}|X_t-\Xb_t|^2\right]\right),
    \end{align*}
    where we used (\ref{drift estimate}) in the second line and (\ref{Y estimate}) in the third line.

    \medskip

    For $\e>0$ small enough, $X_0,\Xb_0$ such that $\E\Bigl[|X_0-\Xb_0|^2\Bigr]=\sum^N_{i=1}W_2^2(\m^i,\overline{\m}^i)$, we derive
    \[\sup_{t\in[0,T]}\sum^N_{i=1}W^2_2(m^i_t,\mb^i_t)\leq\E\left[\sup_{t\in[0,T]}|X_t-\Xb_t|^2\right]\leq C\E\Bigl[|X_0-\Xb_0|^2\Bigr]=C\sum^N_{i=1}W_2^2(\m^i,\overline{\m}^i).\]
    \end{proof}

\begin{corollary}\label{uniqueness cor}
    If $\m=\overline{\m}$, then $(w,m)=(\overline{w},\mb)$. Moreover, $\a^i(t,\cdot)=\alb^i(t,\cdot)$ on the support of $m^i_t$.
\end{corollary}

\begin{proof}
    It follows from the proof of Theorem \ref{stability theorem} that $X_t=\Xb_t$ almost surely, $m_t=\mb_t$. In turn, \eqref{Y estimate} gives $Y_0=\Yb_0$, consequently the left hand side of \eqref{drift estimate} vanishes. Applying Fubini's theorem on (\ref{drift estimate}) gives
    \[0=\E\left[\int^T_0|\a^i(t,X^i_t)-\alb^i(t,\Xb^i_t)|^2\ dt\right]=\int^T_0\int_\Rd|\a^i(t,x)-\alb^i(t,x)|^2\ dm^i_t(x)dt.\]
    Hence $\a^i_t=\alb^i_t$ on the support of $m^i_t$, appealing to the representation formula for $w^i,\overline{w}^i$ as value functions of optimal control problems, we also obtain $w=\overline{w}.$
\end{proof}

\section{The Mean Field Game and the Convergence Results}\label{section MFG}

\subsection{The Mean Field Game}

Naturally, to be able to talk about the convergence problem in this section we need to impose assumptions that are compatible with Assumptions \ref{Assumption L}, \ref{Assumption L loc}, \ref{Assumption semimon} and \ref{assumption DxL}, so that the $N$-player games with costs as in (\ref{MF cost functions}) are wellposed. 
%

\begin{assumption}\label{assumption MF}\
    \begin{enumerate}
    \item $m_0\in\prob_q(\Rd)$ for some $q>2.$ 
    \item $\R^d\times\Rd\times\prob(\Rd\times\Rd)\ni(x,a,\m)\mapsto L(x,a,\m)\in C^1$ are convex functions of the variables $(x,a)$ for all $\m$.
        
        \item $\textnormal{There exists } C>0,\ \textnormal{such that}\ L(x,a,\m)\geq\th_1(|a|)-C$ where $\th_1:[0,\infty)\rightarrow[0,\infty)$ is superlinear.

        \item $\textnormal{There exists }\l_{\max}\geq\l_{\min}>0$ such that $\l_{\min}I\leq D^2_{aa}L\leq\l_{\max}I$ uniformly over all variables.
        \item $L$ is Lipschitz continuous and semiconcave in the variable $x$ uniformly over all other variables. Moreover, for all $R>0,\ |x|, |a|\leq R,\textnormal{there exists }   C_R>0$ such that $ |L(x,a,\m)|\leq C_R(1+W_2(\m,\d_0))$

        \item $g\in C^1$ is convex, Lipschitz continuous and semiconcave in the variable $x$ uniformly.

        \item $D_aL\in C^1$ is Lipschitz continuous in all variables uniformly (w.r.t. $W_2$ in the measure variable).

        \item $D_xL$ is Lipschitz continuous in all variables uniformly (w.r.t. $W_2$ in the measure variable), bounded in the variable $x$ uniformly. To be precise, for all $R>0,\ x\in\Rd, |a|\leq R,\textnormal{there exists }   C_R>0$ such that $ |D_xL(x,a,\m)|\leq C_R$.

        \item $\l_{\min}^{-1} \norm{D_m^aD_aL}_{\infty,\op}<1$ 

    \end{enumerate}
As in Assumption \ref{Assumption L loc}, one can weaken some of the assumptions, whenever $\b=0.$
\end{assumption}
\begin{remark}
For a continuously intrinsically differentiable vector field $F:\prob_2(\Rd\times\Rd)\rightarrow\Rd$ with components $F=(F_1,\ldots,F_d)$, we write $D^a_mF_n(m)(x,a)$ to denote the last $d$ components of the full intrinsic derivative $D_mF_n(m)(x,a)$ and we write $D^a_mF(m)(x,a)$ for the full $d\times d$ matrix.

\medskip

Now for the finite dimensional projection
\[F^N(x^1,\ldots,x^N,a^1,\ldots,a^N)=F\left(\frac{1}{N}\sum_{i=1}^N\d_{x^i,a^i}\right),\]
we have $D_{a^j}F^N(x^1,\ldots,x^N,a^1,\ldots,a^N)=\frac{1}{N}D^a_mF\left(\frac{1}{N}\sum_{i=1}^N\d_{x^i,a^i}\right)(x^j,a^j).$

\medskip

Hence, condition 9 of Assumption \ref{assumption MF} implies condition 9 of Assumption \ref{Assumption L}, when the running costs $L^i$ of the $N$-player game each take the form of (\ref{MF cost functions}).    
\end{remark}

\begin{assumption}
    \label{Assumption MF semimon}
There exist constants $C_{L,a}>0,\ C_{L,x},C_g\geq0$ such that
\begin{equation*}
    C_{disp}:=C_{L,a}-TC_g-\frac{T^2}{2}C_{L,x}>0,
\end{equation*}
 for all $X,\Xb,\a,\alb\in L^2(\Om)$,
    \begin{align*}
        \E\Biggl[& \Bigl(D_aL(X,\a,\Leb(X,\a))-D_aL(\Xb,\alb,\Leb(\Xb,\alb))\Bigr)\cdot(\a-\alb) \\
        &+\Bigl(D_xL(X,\a,\Leb(X,\a))-D_xL(\Xb,\alb,\Leb(\Xb,\alb))\Bigr)\cdot(X-\Xb)\Biggr]\geq C_{L,a}\E\Bigl[|\a-\alb|^2\Bigr]-C_{L,x}\E\Bigl[|X-\Xb|^2\Bigr]
    \end{align*}
and
\[\E\Bigl[\Bigl(D_xg(X,\Leb(X))-D_xg(\Xb,\Leb(\Xb))\Bigr)\cdot(X-\Xb)\Biggr]\geq-C_g\E\Bigl[|X-\Xb|^2\Bigr]\]
\end{assumption}

\begin{lemma}\label{semimon lemma}
    $L^i$ and $g^i$ defined as in \eqref{MF cost functions} with $L$ and $g$ satisfying Assumption \ref{Assumption MF semimon} also satisfy  Assumption \ref{Assumption semimon} provided $N\in\N$ is large enough. In particular, if $\norm{D^2_{mm}L^i},\norm{D^2_{mm}g^i}\leq\g$, then denoting the new semimonotonicity constants of Assumption \ref{Assumption semimon} as $\widetilde{C}_{L,a},\widetilde{C}_g,\widetilde{C}_{L,x}$, we have
    \[\widetilde{C}_{L,a}=C_{L,a}-\frac{\g}{N},\ \widetilde{C}_{L,x}=C_{L,x}+\frac{\g}{N},\ \widetilde{C}_g=C_g+\frac{\g}{N}.\]
\end{lemma}

\begin{proof}
    See \cite[Lemma 3.4]{jackson2025quantitativeconvergencedisplacementmonotone} and its proof.
\end{proof}

The following example is inspired by the example given in \cite[Remark 2.4]{jackson2025quantitativeconvergencedisplacementmonotone}.

\begin{example}\label{example}
Let $L_0:\Rd\times\Rd\times\prob_2(\Rd\times\Rd)\rightarrow\R$ be convex in the first two variables uniformly and bounded below. Furthermore, suppose
\[D_{(x,a)}L_0(x,a,\m)=D_m\sL(\m)(x,a)\]
for some $C^2$ $\varrho$-displacement semiconvex (also called geodesically semiconvex, see \cite{ParkerSomeConvexityCriteria}) $\sL:\prob_2(\Rd\times\Rd)\rightarrow\R$ with bounded first and second intrinsic derivatives. Furthermore, suppose \[D_aD_m^a\sL(\m)(x,a)\geq -\varsigma I,\ \norm{D^2_{mm}\sL}_{\i,\op}\leq \g.\]
Define
\[L(x,a,\m)=\frac{\k}{2}|a|^2+L_0(x,a,\m),\]
for $\k>0$. Then $L$ satisfies all the previous assumptions (except \eqref{C disp} as we have not specified the functions $g$) if 
\[\k-\varsigma>\min\left\{\g,\varrho+\frac{\g}{N}\right\}>0.\]
%
Note that we have used \cite[Theorem 1.2]{ParkerSomeConvexityCriteria} and Lemma \ref{semimon lemma}.
\end{example}

For completeness we also state a wellposedness theorem for the MFG system. We note that the theorem below can be extended to the case $q=2$.
\begin{theorem}\label{existence thm MFG}
    Under Assumption \ref{assumption MF} and \ref{Assumption MF semimon}, there exists a unique solution $(v,\m)$ to the system \eqref{MFG}.
\end{theorem}

\begin{proof}
Let $\b>0$, the existence of \eqref{MFG} with $\b=0$ is postponed until Theorem \ref{thm vanishing viscosity MFG}. The PMP system \eqref{PMPMFG} can be reformulated as a typical McKean--Vlasov FBSDE (c.f. \cite{jackson2025quantitativeconvergencedisplacementmonotone}) using a mapping $\Phi:\prob_2(\Rd\times\Rd)\rightarrow\prob_2(\Rd\times\Rd)$ which has the property
\[\Phi(\Leb(X_t,Y_t))=\Leb(X_t,-D_pH(X_t,Y_t,\Phi(\Leb(X_t,Y_t))).\]
It is proven in \cite[Lemma 3.3, Appendix A]{jackson2025quantitativeconvergencedisplacementmonotone}  that such a mapping $\Phi$ exists, is unique and the PMP system \eqref{PMPMFG} is wellposed in the strong sense.

\medskip

The uniqueness of \eqref{MFG} immediately follows as a solution to the MFG system must associate by Pontryagin's maximum principle a solution to the wellposed PMP system (\ref{PMPMFG}).

\medskip

Given the unique strong solution $(X,Y,Z)$ to \eqref{PMPMFG}, let $v\in C^{1,2}([0,T]\times\Rd)$ be the unique classical solution to 

\[\begin{cases}
    -\p_tv(t,x)-\b\D v(t,x)+H(x,D_xv(t,x),\Phi(\Leb(X_t,Y_t)))=0 , &\textnormal{in } (0,T)\times\Rd,\\
    v(T,x)=g(x,\Leb(X_T)), &\textnormal{in } \Rd,
\end{cases}\]
then $(v,\m)$ is a solution to \eqref{MFG}, where we denoted $\m_t=\Phi(\Leb(X_t,Y_t))$. Indeed, by It\^o's formula the triplet $(X_t,D_xv(t,X_t),\sqrt{2\b}D_{xx}^2v(t,X_t))$ satisfies the FBSDE \eqref{PMPMFG}, by uniqueness we must have $Y_t=D_x v(t,X_t)$. Considering the aforementioned property of the mapping $\Phi$, we find that the remaining equations in \eqref{MFG} are satisfied by $(v,\m)$.
\end{proof}

 \begin{remark}
     We can also directly obtain the existence of solutions to \eqref{MFG} (for all $\b\geq0$ and $q\geq 2$) using a similar argument as in the proof of Theorem \ref{existence thm stocahstic} and  \ref{existence thm intermediate}, i.e.  a fixed point scheme $S(\alb)=\a$ given by 
    \[\a(t,x)=-D_pH(x,D_xv(t,x),\m_t)\]
    where $(v,\m)$ is the unique solution to
    \begin{equation*}
    \begin{cases}
        \p_tm_t-\b\D m_t+\na\cdot(m_t \alb_t)=0, &\textnormal{in } (0,T)\times\Rd,\\
        \m_t=(\textnormal{Id},\alb(t,\cdot))_\sharp m_t, &\textnormal{in } [0,T],
        \\-\p_tv(t,x)-\b\D v(t,x)+H(x,D_xv(t,x),\m_t)=0, &\textnormal{in } (0,T)\times\Rd,\\ m_0=\Leb(\xi),\ \ v(T,x)=g(x,m_T) , &\textnormal{in } \Rd.
    \end{cases}
\end{equation*}
 \end{remark}

\subsection{Convergence to the Mean Field Game}\label{section convergence}

\begin{theorem}\label{conv thm 1}Suppose 
    \begin{itemize}[label=\raisebox{0.25ex}{\tiny$\bullet$}]
    \item Assumption \ref{assumption MF} and \ref{Assumption MF semimon} hold true;
    \item $\b\geq0,N\in\N$; 
         \item   $(\widehat{\xi}^i)_{i=1,\ldots,N}$ are $L^q$-i.i.d. random variables, with common law $\Leb(\widehat{\xi}^i)=m_0\in\prob_q(\Rd)$, for some $q>2$, $q\notin\{4,\frac{d}{d-2}\}$, $i=1,\ldots,N$.
     \end{itemize} 
     Then there exists a constant $C>0$ independent of $N\in\N$ such that
    \[\sup_{i=1,\ldots,N}\E\Biggl[\sup_{t\in[0,T]}|X^i_t-\Xh^i_t|^2+\int^T_0|A^i_t-\Ah^i_t|^2\ dt\Biggr]\leq Cr_{d,q}(N)\]
        \begin{itemize}[label=\raisebox{0.25ex}{\tiny$\bullet$}]
        \item $X=(X^1,\dots,X^N)$ is the first component of the solution to \eqref{PMP} with initial condition $\widehat{\xi} = (\widehat{\xi}^1,\dots,\widehat{\xi}^N)$;
        \item $\Xh=(\Xh^1,\dots,\Xh^N)$ is the first component of the solution to \eqref{iidPMPMFG} with initial condition $\widehat{\xi} = (\widehat{\xi}^1,\dots,\widehat{\xi}^N)$;
        \item $(w,m,\a)$ is the solution to the $N$-player system \eqref{DNash} with $m^i_0=\Leb(\xi^i)\in \prob_q(\Rd)$;
        \item $(v,\m)$ is the solution to the MFG system \eqref{MFG};
        \item $A,\Ah$ are the corresponding control processes given by $A^i_t=\a^i(t,X^i_t)$ and $ \Ah_t^i=-D_pH(\Xh^i_t,\Yh^i_t,\m_t)$.
    \end{itemize} 
    
\medskip

    We also denoted $r_{d,q}(N)$ as the Fournier--Guillin rate \cite{Fournier2013OnTR}
\[r_{d,q}(N)=\begin{cases}
    N^{-\frac{1}{2}}+N^{-\frac{q-2}{q}}\ &d<4\\
    N^{-\frac{1}{2}}\log(1+N)+N^{-\frac{q-2}{q}}\ &d=4\\
    N^{-\frac{2}{d}}+N^{-\frac{q-2}{q}}\ &d>4.
\end{cases}\]
\end{theorem}

\begin{proof}
    Again we produce an estimate on the difference of the drifts by employing semimonotonicity, we use Lemma \ref{Dudley lemma} and \ref{semimon lemma} (note that we still rely on the monotonicity condition from Assumption \ref{Assumption semimon}).
  \begin{align*}
       -&\sum_{i=1}^N\E\Bigl[(A^i_t-\Ah^i_t)\cdot(Y^i_t-\Yh^i_t)\Bigr]+\E\biggl[\Bigl(D_x\Ht(X^i_t,Y^i_t,m^{-i}_t,\a^{-i}_t)-D_xH(\Xh^i_t,\Yh^i_t,\m_t)\Bigr)\cdot(X^i_t-\Xh^i_t)\biggr]\\
       =&\sum_{i=1}^N\E\Biggl[(A^i_t-\Ah^i_t)\cdot\Bigl(D_a\Lt(X^i_t,A^i_t,m^{-i}_t,\a^{-i}_t)-D_aL(\Xh^i_t,\Ah^i_t,\m_t)\Bigr)\\
       &\ \ \ \ \ \ \ \ \ \ \ \ \  \ \ \ +\Bigl(D_x\Lt(X^i_t,A^i_t,m^{-i}_t,\a^{-i}_t)-D_xL(\Xh^i_t,\Ah^i_t,\m_t\bigr)\Bigr)\cdot(X^i_t-\Xh^i_t)\Biggr]\\
       =&\sum_{i=1}^N\E\Biggl[\E\Biggl[(A^i_t-\Ah^i_t)\cdot\Bigl(D_aL(X^i_t,A^i_t,m^{N,-i}_{X_t,A_t})-D_aL(\Xh^i_t,\Ah^i_t,\m_t)\Bigr)\\
        &\ \ \ \ \ \ \ \ \ \ \ \ \  \ \ \ 
        +\Bigl(D_xL(X^i_t,A^i_t,m^{N,-i}_{X_t,A_t})-D_xL(\Xh^i_t,\Ah^i_t,\m_t\bigr)\Bigr)\cdot(X^i_t-\Xh^i_t)\biggr|\sF^i_t\Biggr]\Biggr]\\
       =&\sum_{i=1}^N\E\Biggl[(A^i_t-\Ah^i_t)\cdot\Bigl(D_aL(X^i_t,A^i_t,m^{N,-i}_{X_t,A_t})-D_aL(\Xh^i_t,\Ah^i_t,\m_t)\Bigr)\\&\ \ \ \ \ \ \ \ \ \ \ \ \  \ \ \ 
        +\Bigl(D_xL(X^i_t,A^i_t,m^{N,-i}_{X_t,A_t})-D_xL(\Xh^i_t,\Ah^i_t,\m_t\bigr)\Bigr)\cdot(X^i_t-\Xh^i_t)\Biggr]
        \\ \geq&\E\Biggl[C_{L,a}|A_t-\Ah_t|^2-C_{L,x}|X_t-\Xh_t|^2-C_{DL}\sum_{i=1}^N\left(\Bigl(|A^i_t-\Ah^i_t|+|X^i_t-\Xh^i_t|\Bigr)W_2\left(\m_t,m^{N,-i}_{\Xh_t,\Ah_t}\right)\right)\Biggr]\\
        \geq&\E\Biggl[(C_{L,a}-\d)|A_t-\Ah_t|^2-(C_{L,x}+\d)|X_t-\Xh_t|^2-C_\d C_N\Biggr]
    \end{align*}
where we used Young's inequality in the last line with a small parameter $\d>0$ to be fixed later, the quantity $C_N$ is given by
\begin{equation}\label{CN}
C_N=\sum_{i=1}^N\E\biggl[W_2^2\left(\m_t,m^{N,-i}_{\Xh_t,\Ah_t}\right)\biggr]    
\end{equation}
and $C_{DL},$ $C_\d$ are constants independent of $N\in\N$. Note that we used the fact $\Leb(\Xh^i_t,\Ah^i_t)=\m_t$ for all $i=1,\ldots,N,$ which follows from (\ref{compatibility condition}).

\medskip

Similarly,
\begin{align*}
    &\E\left[\bigl(D_x\gt(X^i_T,m^{-i}_T)-D_xg(\Xh^i_T,m_T)\bigr)\cdot(X^i_T-\Xh^i_T)\right]\\
=&\E\left[\E\biggl[\bigl(D_xg(X^i_T,m_{X_T}^{N,i})-D_xg(\Xh^i_T,m_T)\bigr)\cdot(X^i_T-\Xh^i_T)\Bigr|\sF^i_t\biggr]\right]\geq\E\Bigl[-(C_{g}+\d)|X^i_T-\Xh^i_T|^2-C_\d C_N\Bigr].
\end{align*}
Using It\^o's formula on the quantity $\sum^N_{i=1}(Y^i_t-\Yh^i_t)\cdot(X^i_t-\Xh^i_t)$ on the whole time interval $t\in[0,T]$, we derive the estimate on the difference of the drifts 
\begin{align*}
    &\E\left[(C_{L,a}-\d)\int^T_0|A_t-\Ah_t|^2\ dt\right]\\
    \leq&\E\left[(C_g+\d)|X_T-\Xh_T|^2+(C_{L,x}+\d)\int^T_0|X_t-\Xh_t|^2\ dt\right]+C_\d C_N.    
\end{align*}
Integrating over $[0,T]$ the inequality
\[\E\Bigl[|X_t-\Xh_t|^2\Bigr]\leq\E\Biggl[t\int^t_0|A_s-\Ah_s|^2\ ds\Biggr]\]
we can further estimate 
\begin{align*}
    \E\left[(C_{L,a}-\d)\int^T_0|A_t-\Ah_t|^2\ dt\right]
    \leq\E\left[\left( T(C_g+\d)+\frac{T^2}{2}(C_{L,x}+\d)\right)\int^T_0|A_t-\Ah_t|^2\ dt\right]+C_\d C_N .   
\end{align*}
Which implies
\begin{align*}
    \E\left[\left(C_{disp}-\d\left(1+T+\frac{T^2}{2}\right)\right)\int^T_0|A_t-\Ah_t|^2\ dt\right]
    \leq C_\d C_N .   
\end{align*}
Upon choosing $\d>0$ small enough, arguing analogously as \cite[Theorem 3.2.2]{Zhang2017} we find
\begin{equation}\label{C d estimate}
    \E\Biggl[\sup_{t\in[0,T]}|X_t-\Xh_t|^2+\int^T_0|A_t-\Ah_t|^2\ dt\Biggr]\leq C_\d C_N.
\end{equation}

The quantity $C_N$ can be estimated using \cite[Theorem 1]{Fournier2013OnTR} and \cite[(5.1)]{jackson2025quantitativeconvergencedisplacementmonotone}
    \begin{align*}
C_N&=\sum_{i=1}^N\E\biggl[W_2^2\left(\m_t,m^{N,-i}_{\Xh_t,\Ah_t}\right)\biggr]\\
    &\leq 2\sum_{i=1}^N\E\biggl[W_2^2\left(\m_t,m^{N}_{\Xh_t,\Ah_t}\right)+W_2^2\left(m^{N}_{\Xh_t,\Ah_t},m^{N,-i}_{\Xh_t,\Ah_t}\right)\biggr]\\
    &\leq 2NCr_{d,q}(N)+\frac{2}{N(N-1)}\sum_{i=1}^N\E\left[\sum_{j\neq i}\left(|\Xh^i_t-\Xh^j_t|^2+|\Ah^i_t-\Ah^j_t|^2\right)\right]\\
    &\leq2NCr_{d,q}(N)+2\left(\int_{\Rd\times\Rd}(|x|^q+|a|^q)\ d\m_t(x,a)\right)^\frac{2}{q}\\
    &\leq \widehat{C}Nr_{d,q}(N),
\end{align*}
where $\widehat{C}>0$ is a constant independent of $N\in\N$.

\medskip

As $X_t^i$ are exchangeable random variables, appealing to the uniqueness of distributed equilibria given by Corollary \ref{uniqueness cor} one can argue as in \cite[Lemma 3.12]{jacksontangpi} to conclude the proof.
\end{proof}

\begin{remark}
    While one can derive an \say{uniform in $N$ stability} extension of Theorem \ref{stability theorem} analogous to \cite[Proposition 4.1]{jackson2025quantitativeconvergencedisplacementmonotone}, it seems not possible to argue as in \cite[Section 5]{jackson2025quantitativeconvergencedisplacementmonotone} to obtain the convergence result above, as we cannot directly compare the functions $D_p\Ht$ and $D_pH.$ 

    \medskip

    For the same reasoning one cannot directly compare distributed equilibria and open-loop equilibria as done in \cite[Section 5.2]{cirant2025nonasymptoticapproachstochasticdifferential}, at least when the Hamiltonian is nonseparable. However, one can now easily derive a bound on the difference (in the sense of the theorem above) between distributed and closed/open-loop equilibria via comparing both to mean field equilibria.    
\end{remark}

We can extend the previous theorem to the non-exchangeable case when the initial conditions of the $N$-player system do not necessarily match those of the i.i.d. mean field system (\ref{iidPMPMFG}), combining the techniques of Theorem \ref{stability theorem} and \ref{conv thm 1}, we only sketch the proof to avoid repetition.
\begin{theorem}\label{conv thm2}Suppose 
       \begin{itemize}[label=\raisebox{0.25ex}{\tiny$\bullet$}]
       \item Assumption \ref{assumption MF} and \ref{Assumption MF semimon} hold true;
       \item  $\b\geq0, N\in\N;$
         \item  $(\xi^i)_{i=1,\ldots,N}$ are $L^q$-independent random variables;
         \item   $(\widehat{\xi}^i)_{i=1,\ldots,N}$ are $L^q$-i.i.d. random variables, with common law $\Leb(\widehat{\xi}^i)=m_0\in\prob_q(\Rd)$, for some $q>2$, $q\notin\{4,\frac{d}{d-2}\}$, $i=1,\ldots,N$.
     \end{itemize}
     Then there exists a constant $C>0$ independent of $N\in\N$ such that
    \[\frac{1}{N}\sum^N_{i=1}\E\Biggl[\sup_{t\in[0,T]}|X^i_t-\Xh^i_t|^2+\int^T_0|A^i_t-\Ah^i_t|^2\ dt\Biggr]\leq C(K(N) +r_{d,q}(N))\]
    where \begin{itemize}[label=\raisebox{0.25ex}{\tiny$\bullet$}]
        \item $X=(X^1,\dots,X^N)$ is the first component of the solution to \eqref{PMP} with initial condition $\xi = (\xi^1,\dots,\xi^N)$;
        \item $\Xh=(\Xh^1,\dots,\Xh^N)$ is the first component of the solution to \eqref{iidPMPMFG} with initial condition $\widehat{\xi} = (\widehat{\xi}^1,\dots,\widehat{\xi}^N)$;
        \item $(w,m,\a)$ is the solution to the $N$-player system \eqref{DNash} with $m^i_0=\Leb(\xi^i)\in \prob_q(\Rd)$;
        \item $(v,\m)$ is the solution to the MFG system \eqref{MFG};
        \item  $A,\Ah$ are the corresponding control processes given by $A^i_t=\a^i(t,X^i_t)$ and $ \Ah_t^i=-D_pH(\Xh^i_t,\Yh^i_t,\m_t)$. 
    \end{itemize}  and 
    \[K(N):=\frac{1}{N}\sum_{i=1}^N W_2^2(m_0,\Leb(\xi^i_N)).\]
    %
\end{theorem}

\begin{proof}[Sketch of Proof]
   Following the proof of Theorem \ref{conv thm 1}, when we apply It\^o's formula on the quantity $\sum^N_{i=1}(Y^i_t-\Yh^i_t)\cdot(X^i_t-\Xh^i_t)$ on the whole time interval $t\in[0,T]$, we derive instead the following estimate on the difference of the drifts
\begin{align*}
    &\E\left[(C_{L,a}-\d)\int^T_0|A_t-\Ah_t|^2\ dt\right]\\
    \leq&\E\left[(Y_t-\Yh_t)\cdot(X_t-\Xh_t)+(C_g+\d)|X_T-\Xh_T|^2+(C_{L,x}+\d)\int^T_0|X_t-\Xh_t|^2\ dt\right]+C_\d C_N.    
\end{align*}
In turn \eqref{C d estimate} becomes (after applying Young's inequality with parameter $\e>0$ to be fixed)
\begin{equation}\label{C d estimate 2}
    \E\Biggl[\sup_{t\in[0,T]}|X_t-\Xh_t|^2+\int^T_0|A_t-\Ah_t|^2\ dt\Biggr]\leq \E\Bigl[\e|Y_0-\Yh_0|^2+C_\e|X_0-\Xh_0|^2\Bigr]+C_\d C_N.
\end{equation}
Now following similar arguments as in the proof of Theorem \ref{stability theorem}, for $\e>0$ small enough, we arrive at the estimate, where the constant $C>0$ is independent of $N\in\N$
    \begin{equation*}
        \E\Bigl[|Y_0-\Yh_0|^2\Bigr]\leq\frac{1}{1-C\e}\left(C_\e\E\Bigl[|X_0-\Xh_0|^2\Bigr]+C\E\left[\sup_{t\in[0,T]}|X_t-\Xh_t|^2\right]+C_\d C_N\right).
    \end{equation*}
Plugging the inequality above into \eqref{C d estimate 2} and take $\e>0$ small enough yields the desired inequality, noting that 
\[K(N)=\frac{1}{N}\sum_{i=1}^N W_2^2(m_0,\Leb(\xi^i_N))=\frac{1}{N}\E\Bigl[|X_0-\Xh_0|^2\Bigr].\]
\end{proof}

Next we give a concentration inequality based on \cite[Theorem 2]{Fournier2013OnTR}, the proof of which is essentially the same as the proof of \cite[Theorem 13]{LauriereConvMFGC}.

\begin{corollary}\label{Corollart concentration}
In the setting of the previous theorem, suppose further that either $q>4$ or there is $\s\neq 2,\g>0$ such that $\mathcal{E}_{\s,\g}(m_0)<\i$, where
\[\mathcal{E}_{\s,\g}(m_0)=\int_\Rd e^{\g|x|^\s}\ dm_0(x).\]
Then there exist constants $C,c>0$ independent of $N\in\N,\e>0$ such that for all $\e>0$
    \[\mathbb{P}\left(\sup_{t\in[0,T]} W_2^2(m^N_{X_t},\Leb(\Xh_t))>\e\right)\leq C\left(\frac{1}{\e}\Bigl(K(N)+r_{d,q}(N)\Bigr)+a_{\e}(N)\mathds{1}_{\{\e\leq2\}}+b_{\e}(N)\right),\]
    where
    \[a_{\e}(N)=\begin{cases}
        \exp(-cN\e^2) &d<4\\
        \exp\left(-cN\e^2\left(\log(1+\frac{1}{\e})\right)^{-2}\right) &d=4\\
        \exp(-cN\e^{\frac{d}{2}}) &d>4
    \end{cases}\]
    and
    \[b_{\e}(N)=\begin{cases}
        N(N\e)^{\frac{-(q-\e)}{2}} &q>4,\e<4\\
        \exp(-cN\e^\frac{\s}{2})\mathds{1}_{\{\e>2\}} &\a>2,\mathcal{E}_{\s,\g}(m_0)<\i\\
        \exp(-c(N\e)^\frac{\s-\e}{2})\mathds{1}_{\{\e\leq2\}}+\exp(-c(N\e)^\frac{\s}{2})\mathds{1}_{\{\e>2\}} &\e<\s<2, \mathcal{E}_{\s,\g}(m_0)<\i.
    \end{cases}\]
\end{corollary}

\begin{proof}
    \begin{align*}
    &\PR\left(\sup_{t\in[0,T]} W^2_2(m^N_{X_t},\Leb(\Xh_t))>\e\right)\\
    \leq\ &\PR\left(\sup_{t\in[0,T]} W^2_2\Bigl(m^N_{X_t},m^N_{\Xh_t}\Bigr)>\frac{\e}{2}\right)+\PR\left(\sup_{t\in[0,T]} W^2_2(m^N_{\Xh_t},\Leb(\Xh_t))>\frac{\e}{2}\right)\\
    \leq\ &\frac{2}{\e}\frac{1}{N}\sum_{i=1}^N\E\left[\sup_{t\in[0,T]}|X^i_t-\Xh^i_t|^2\right]+C\Bigl(a_{\e}(N)\mathds{1}_{\{\e\leq2\}}+b_{\e}(N)\Bigr)\\
    \leq\ &C\left(\frac{1}{\e}\Bigl(K(N)+r_{d,q}(N)\Bigr)+a_{\e}(N)\mathds{1}_{\{\e\leq2\}}+b_{\e}(N)\right)
    \end{align*}
    where $C>0$ is a constant independent of $N\in\N$. We used Chebyshev's inequality and \cite[Theorem 2]{Fournier2013OnTR} to estimate the second line, Theorem \ref{conv thm2} to estimate the first term of the third line.
\end{proof}

We can exert a more precise estimate on the value functions. In particular, we can estimate the difference between $w^i$ and $v$ individually, even though we only have an estimate of the average of the difference of the trajectories/drifts. 

\begin{theorem}\label{conv thm 3}
Suppose 
       \begin{itemize}[label=\raisebox{0.25ex}{\tiny$\bullet$}]
       \item Assumption \ref{assumption MF} and \ref{Assumption MF semimon} hold true;
       \item $\b\geq0, N\in\N;$
       \item $(w,m,\a)$ is the solution to the $N$-player system \eqref{DNash} with $m^i_0=\Leb(\xi_N^i)\in\prob_q(\Rd);$
         \item  $(v,\m)$ is the solution to the MFG system \eqref{MFG} with initial condition $m_0\in\prob_q(\Rd),\ q>2,\ q\notin\{4,\frac{d}{d-2}\}$;
         \item  $L,g$ are Lipschitz continuous in the measure variable with respect to the 2-Wasserstein topology.
     \end{itemize}
     Then there exists a constant $C>0$ independent of $N\in\N$ such that
    \[\max_{i=1,\ldots,N}\norm{w^i-v}_\infty\leq C\left(K(N)+r_{d,q}(N)\right)^\frac{1}{2}\]
    where we recall the quantities $K(N),r_{d,q}(N)$ are defined as in Theorems \ref{conv thm 1} and \ref{conv thm2}.
\end{theorem}

\begin{proof}
    Let us first consider the case $\b>0$, recall the identity
    \[H(x,D_xv,\m_t)=D_xv\cdot D_pH(x,D_xv,\m_t)-L(x,-D_pH(x,D_xv,\m_t),\m_t).\]
    By definition of the Hamiltonian we also have for all $a\in\Rd$
    \[\Ht(x,D_xw^i,m^{-i}_t,\a^{-i}_t)\geq-D_xw^i\cdot a-\Lt(x,a,m^{-i}_t,\a^{-i}_t).\]
    Choosing $a=-D_pH(x,D_xv,\m_t)$ we can bound the difference of the Hamiltonians
    \begin{align*}
        &H(x,D_xv,\m_t)-\Ht(x,D_xw^i,m^{-i}_t,\a^{-i}_t)\\
        \leq& (D_xv-D_xw^i)\cdot D_pH-\E\biggl[L(x,-D_pH,\m_t)-L\left(x,-D_pH,m^{N,-i}_{X_t,A_t}\right)\biggr] \\
        \leq&(D_xv-D_xw^i)\cdot D_pH+\E\biggl[\norm{D_mL}_\infty W_2\left(\m_t,m^{N,-i}_{X_t,A_t}\right)\biggr]
    \end{align*}
    where $X,\Xh$ are the first components of the FBSDEs \eqref{PMP},\eqref{iidPMPMFG} respectively, $A,\Ah$ are the corresponding control processes. For $t\in[0,T)$ we define
    \begin{align*}      C(t):=&\E\biggl[\norm{D_mL}_\infty W_2\left(\m_t,m^{N,-i}_{X_t,A_t}\right)\biggr]\\
\leq&\E\biggl[\norm{D_mL}_\infty\left(W_2\left(\m_t,m^{N,-i}_{\Xh_t,\Ah_t}\right)+W_2\left(m^{N,-i}_{X_t,A_t},m^{N,-i}_{\Xh_t,\Ah_t}\right)\right)\biggr]
\\ \leq&\norm{D_mL}_\infty\left(\widehat{C}r_{d,q}(N)^\frac{1}{2}+\left(\frac{1}{N-1}\sum_{j\neq i}\E\left[|X^j_t-\Xh^j_t|^2+|A^j_t-\Ah^j_t|^2\right]\right)^\frac{1}{2}\right).\\
    \end{align*}
Next, we bound the difference between the final costs
\begin{align*}
\gt(x,m^{-i}_T)-g(x,m_T)
    &\geq- \norm{D_mg}_\i\left(\widehat{C}r_{d,q}(N)^\frac{1}{2}+\E\biggl[W_2\left(m^{N,-i}_{X_T},m^{N,-i}_{\Xh_T}\right)\biggr]
    \right)\\ &\geq-\norm{D_mg}_\i\left(\widehat{C}r_{d,q}(N)^\frac{1}{2}+\left(\frac{1}{N-1}\sum_{j\neq i}\E\left[|X^j_T-\Xh^j_T|^2\right]\right)^\frac{1}{2}\right).
\end{align*}
We set
\[C_T:=\norm{D_mg}_\i\left(\widehat{C}r_{d,q}(N)^\frac{1}{2}+\left(\frac{1}{N-1}\sum_{j\neq i}\E\left[|X^j_T-\Xh^j_T|^2\right]\right)^\frac{1}{2}\right).\]

Define $\Psi(t,x):=w^i(T-t,x)-v(T-t,x)+C_T+\int^{T-t}_0 C(s)\ ds$, considering the respective HJB equations $w^i,v$ solve and the bound on the difference of the Hamiltonians we derived above, we find that $u$ is a subsolution to the linear parabolic equation
\[\begin{cases}
    \p_t\Psi(t,x)-\b\D \Psi(t,x)+D_pH(x,D_xv(T-t,x),\m_{T-t})D_x\Psi(t,x)\leq 0, &\textnormal{in } (0,T)\times\Rd,\\
    \Psi(0,x)\geq0, &\textnormal{in } \Rd.
\end{cases}\]

%

Therefore, by the maximum principle \cite[Section 2.4, Theorem 9]{friedman2013partial} (note that condition (4.2) therein is satisfied as $w^i,v$ grow at most quadratically), we have $\Psi(t,x)\geq0$ for all $(t,x)\in[0,T]\times\Rd$, i.e.
\[w^i(t,x)-v(t,x)\geq -\widehat{C}\left(r_{d,q}(N)^\frac{1}{2}+\left(\frac{1}{N-1}\E\Biggl[\sup_{t\in[0,T]}|X_t-\Xh_t|^2+\int^T_0|A_t-\Ah_t|^2\ dt\Biggr]\right)^\frac{1}{2}\right).\]
By Theorem \ref{conv thm2} we have 
\begin{equation}\label{C hat}
    v(t,x)-w^i(t,x)\geq- \widehat{C}\left(K(N)+r_{d,q}(N)\right)^\frac{1}{2}
\end{equation}
where $\widehat{C}$ is independent of $N\in\N$. Reversing the roles of $v,w^i$ and $H,\Ht$ one can derive the corresponding bound from above.

\medskip

The deterministic case $\b=0$ can be proven instead using the comparison principle for Hamilton-Jacobi equations. Following some of the bounds derived above, the function
\[w^i(t,x)+\widehat{C}(K(N)+r_{d,q}(N))^\frac{1}{2})(1+t)\]
is a supersolution of the Hamilton-Jacobi equation in \eqref{MFG} (with $\b=0)$, where $\widehat{C}$ is the same constant as the one appearing in \eqref{C hat}, while
\[v(t,x)+\widehat{C}(K(N)+r_{d,q}(N))^\frac{1}{2})(1+t)\]
is a supersolution of the Hamilton-Jacobi equation in \eqref{DNash deterministic}. Appealing to the comparison principle we derive the result also in the case $\b=0.$
\end{proof}



   
\begin{remark} 
    Due to the inability to compare $D_pH$ and $D_p\Ht$, we are unable obtain any quantitative convergence result for the gradients $D_xw^i$ and $D_xv$ in the stochastic case when there is nontrivial interaction through controls. However, we expect that qualitative convergence from compactness methods still holds for gradients, akin to results stated in 
    \cite[Theorem 2.3]{MFGlionslasry}.
\end{remark}

\section{Vanishing Viscosity Limit and Gradient Convergence for Deterministic Problems}\label{sec:vanishing}

Throughout this section we impose Assumption \ref{assumption MF}. Note that in this section it is \textit{not} possible to relax Lipschitz assumptions on $L,g$ to a similar form as in Assumption \ref{Assumption L loc}, even if a theorem herein only deal with the deterministic case $\b=0$.

\medskip

We aim to prove in this section a version of Theorem \ref{conv thm 3} in the case of the fully deterministic games $\b=0$ and $m^i_0=\d_{z^i}$, but for the gradients $D_xw$ and $D_xv.$ However, we first need to take a detour and study the vanishing viscosity limit of the stochastic MFG system, which we recall takes the form, for $\b\geq0$ 

\begin{equation}
\left\{    
    \begin{array}{ll}
        -\p_tv_\b(t,x)-\b\D v_\b(t,x)+H(x,D_xv_\b(t,x),\m^\b_t)=0, & {\rm in}\ (0,T)\times\Rd,\\[3pt]
        \p_tm^\b_t-\b\D m^\b_t-\na\cdot(m^\b_tD_p H(x,D_xv_\b(t,x),\m^\b_t))=0, & {\rm in}\ (0,T)\times\Rd,\\[3pt]
        \m^\b_t=(\textnormal{Id},-D_pH(\cdot,D_xv_\b(t,\cdot),\m^\b_t))_\sharp m_t^\b, &\textnormal{in }[0,T],\\[3pt]
         v_\b(T,x)=g(x,m^\b_T),\ m^\b_0=m_0, & {\rm in}\ \Rd,\\
    \end{array}
\right.
\end{equation}

\medskip
We proved in Theorem \ref{existence thm MFG}, the PDE system above is wellposed for $\b>0$. Let us recall the fixed point mapping $\Phi$ briefly used in the proof of Theorem \ref{existence thm MFG}. It was proven in \cite[Lemma 3.3]{jackson2025quantitativeconvergencedisplacementmonotone} that there exists a unique mapping $\Phi:\prob_2(\Rd\times\Rd)\rightarrow\prob_2(\Rd\times\Rd)$ which is Lipschitz w.r.t. $W_2$ and has the property, for all $L^2$ random variables $X,Y$
\[\Phi(\Leb(X,Y))=\Leb(X,-D_pH(X,Y,\Phi(\Leb(X,Y))).\]
With the help of the mapping $\Phi$, we can rewrite the stochastic MFG system in the form
\begin{equation}
\left\{    
    \begin{array}{ll}\label{b-MFG}
    -\p_tv_\b(t,x)-\b\D v_\b(t,x)+H\left(x,D_xv_\b(t,x),\Phi\left((\textnormal{Id},D_xv_\b(t,\cdot))_\sharp m_t^\b\right)\right)=0, & {\rm in}\ (0,T)\times\Rd,\\[3pt]
        \p_tm^\b_t-\b\D m^\b_t-\na\cdot\left(m^\b_tD_p H\left(x,D_xv_\b(t,x),\Phi\left((\textnormal{Id},D_xv_\b(t,\cdot))_\sharp m_t^\b\right)\right)\right)=0, & {\rm in}\ (0,T)\times\Rd,\\[3pt]
        v_\b(T,x)=g(x,m^\b_T),\ m^\b_0=m_0, & {\rm in}\ \Rd,\\
    \end{array}
\right.
\end{equation}

\medskip

Similarly, the PMP FBSDE can be rewritten as
\begin{equation}\label{b-PMPMFG}
 \left\{   
    \begin{array}{ll}
      X^\b_t=\xi-\displaystyle\int^t_0 D_pH\left(X^\b_s,Y^\b_s,\Phi\left(\Leb(X^\b_s,Y^\b_s)\right)\right) \ ds+\sqrt{2\b} B_t, & t\in [0,T]\\[3pt]
        Y^\b_t=D_xg(X^\b_T,\Leb(X^\b_T))-\displaystyle\int^T_tD_xH\left(X^\b_s,Y^\b_s,\Phi\left(\Leb(X^\b_s,Y^\b_s)\right)\right)\ ds-\int^T_tZ^\b_s\ dB_s, & t\in [0,T]
    \end{array}
 \right.
\end{equation}

with $\Leb(X^\b_t)=m^\b_t,$ $Y^\b_t=D_xv_\b(t,X^\b_t)$ and $Z^\b_t=\sqrt{2\b}D^2_{xx}v_\b(t,X_t^\b)$.

\medskip

Let us emphasise again that we impose Assumption \ref{assumption MF} and \ref{Assumption MF semimon} \textit{throughout} this section.
\begin{theorem}\label{thm vanishing viscosity MFG}
    Let $\b=0$, $m_0\in\prob_q(\Rd)$ for some $q>2$, the PDE system \eqref{b-MFG} is wellposed. Furthermore, for all $R>0$
    \[\lim_{\b\rightarrow0}\Bigl(\norm{v_\b-v}_{L^\i([0,T]\times B_R)}+\norm{D_xv_\b-D_xv}_{L^\i([0,T]\times B_R)}\Bigr)=0,\]
    \[\lim_{\b\rightarrow0}\sup_{t\in[0,T]}\left(W_2^2(\nu^{\b}_t,\nu_t)+W^2_2(m^{\b}_t,m_t)\right)=0\]
    where $(v_\b,m^\b)$ is the unique solution to \eqref{b-MFG} with $\b\geq0$, we defined $\nu^\b_t:=(\textnormal{Id},D_xv_{\b}(t,\cdot))_\sharp m_t^{\b}$ and we denoted $(v_0,m^0)$ as $(v,m)$.
\end{theorem}

\begin{proof}
    We take $\b\in(0,1)$. Let us first note that by classical estimates from control theory, we have the following bounds on the value function
    \[\norm{D_xv}_\i\leq C_0\ \ \textnormal{and}\ \ \norm{D^2_{xx}v}_{\i,\op}\leq C_1,\]
    the constants $C_0,C_1>0$ can be made to be independent of $\b$ (e.g. \cite[Appendix A]{cirant2025longtimebehaviorstabilization}). Moreover, following similar arguments as in the proof of Theorem \ref{existence thm stocahstic}, we also have 
    \[\norm{\p_tv}_{L^\i([0,T]\times B_R)}<C(R) \]
and equi-Hölder continuity in time for $D_xv$ (the Hölder constant can be made to be independent of $\b\in(0,1)$, e.g. see the precise statement of \cite[Lemma B.1]{CirantRedaelli}).

    \medskip

    By Arzelà--Ascoli theorem, we can extract a subsequence $\{\b_j\}$ such that  $v_{\b_j}$ converges to  some function $v\in C^{1,1}([0,T]\times\Rd)$ locally uniformly, with $D_xv_{\b_j}$ also converging locally uniformly to $D_xv.$ That is,
    \[\lim_{\b_j\rightarrow0}\Bigl(\norm{v_{\b_j}-v}_{L^\i([0,T]\times B_R)}+\norm{D_xv_{\b_j}-D_xv}_{L^\i([0,T]\times B_R)}\Bigr)=0.\]
    We aim to show that $v$ is a classical solution to \eqref{b-MFG} with $\b=0$. We note that the proof of uniqueness in Theorem \ref{existence thm MFG} also extend to the case $\b=0$, as the proof of uniqueness of the PMP system \eqref{PMPMFG} in \cite[Appendix A]{jackson2025quantitativeconvergencedisplacementmonotone} also holds in this case. 

    \medskip

Denote by $m\in C([0,T];\prob_2(\Rd))$ the unique distributional solution to the following continuity equation along with initial condition $m_0$,
\[\p_tm_t-\na\cdot\left(m_tD_p H\left(x,D_xv(t,x),\Phi\left((\textnormal{Id},D_xv(t,\cdot))_\sharp m_t\right)\right)\right)=0,\]
which exists by Theorem \ref{thm cont eqn}, note that the conditions to use the aforementioned theorem are verified. Indeed, the Lipschitz continuity in $x$ follows from the Lipschitz continuity of $D_pH$ and $D_xv$, the Lipschitz continuity in the measure variable w.r.t. $W_2$ follows from the Lipschitz continuity of $\Phi$ w.r.t. $W_2$ and the fact that for all Lipschitz functions $f$ we have
\[W_2(f_\sharp \m, f_\sharp \nu)\leq \norm{f}_\i W_2(\m,\nu).\]
Moreover, by Theorem \ref{legendre} the vector field $V(t,x,\m)=D_p H\left(x,D_xv(t,x),\Phi\left((\textnormal{Id},D_xv(t,\cdot))_\sharp \mu\right)\right)$ is in fact bounded.

\medskip

Let $X$ be the unique strong solution to the following SDE
\[X_t=\xi-\int^t_0 D_pH\left(X_s,D_xv(s,X_s),\Phi\Bigl((\textnormal{Id},D_xv(t,\cdot))_\sharp m_t\Bigr)\right) ds.\]
%

    %


   We need to show $\nu_t^{\b_j}:=(\textnormal{Id},D_xv_{\b_j}(t,\cdot))_\sharp m_t^{\b_j}\rightarrow \nu_t:=(\textnormal{Id},D_xv(t,\cdot))_\sharp m_t$ in $W_2$, $\beta_j\to 0$. First, let us assume that $m_0$ is compactly supported, equivalently that $\xi\in L^\i(\Om)$, observe that
\begin{align*}
W_2^2(\nu^{\b_j}_t,\nu_t)+W^2_2(m^{\b_j}_t,m_t)&\leq\E\Bigl[2|X^{\b_j}_t-X_t|^2+|D_xv_{\b_j}(t,X_t^{\b_j})-D_xv(t,X_t)|^2\Bigr]\\
&\leq \E\Bigl[2(1+C_1)|X^{\b_j}_t-X_t|^2+|D_xv_{\b_j}(t,X_t)-D_xv(t,X_t)|^2\Bigr]\\
&\leq2\left((1+C_1)\E\Bigl[|X^{\b_j}_t-X_t|^2\Bigr]+\norm{D_xv_{\b_j}(t,\cdot)-D_xv(t,\cdot)}_{L^\i(B_R)}^2\right)
   \end{align*}
   where $R>0$ is large enough so that \cite[Lemma 3.1]{vanishingviscosityMFG} applies. By a classical application of Grönwall's inequality to (McKean--Vlasov) SDEs (see also \cite[Lemma 3.2]{vanishingviscosityMFG}), using the fact that $\Phi$ is Lipschitz w.r.t. $W_2$, we can derive the bound
   \[\sup_{t\in[0,T]}\E\Bigl[|X^{\b_j}_t-X_t|^2\Bigr]\leq C\left(\b_j+\int^T_0 \norm{D_xv_{\b_j}(s,\cdot)-D_xv(s,\cdot)}_{L^\i(B_R)}^2\ ds\right)\rightarrow 0\]
   as $\b_j\rightarrow0$ (note that the constant $C>0$ is independent of $\b_j>0$). Consequently $\nu_t^{\b_j}\rightarrow \nu_t$ in $W_2$ for all $t\in[0,T]$ and passing to the limit in the Hamilton-Jacobi equation in \eqref{b-MFG}, we deduce that $(v,m)$ is a solution to \eqref{b-MFG} with $\b=0$. By uniqueness of the PDE system \eqref{b-MFG} with $\b=0$, we have that necessarily the whole sequence converges and the proof is complete in the case that $m_0$ is compactly supported. 

   \medskip

   In the general case $m_0\in\prob_q(\Rd)$, one can employ the approximation argument of \cite[Corollary 3.6]{vanishingviscosityMFG}.

   %
\end{proof}

For initial conditions $z\in(\Rd)^N$, we introduce the following \say{intermediate} system, for which the existence of classical solutions can be shown by the same method as in Theorem \ref{existence thm stocahstic}. Such a system is constructed to facilitate gradient estimates that pass to the vanishing viscosity limit which would otherwise be difficult to obtain without the use of parabolicity and hence maximum principle. It is essential that we only regularise the HJB equation and not the \say{state equations} to keep the \say{states} deterministic, so that we do not have to lift the Hamiltonian (as in general it is not possible to directly compare $D_p\Ht$ and $D_pH$).
\begin{equation*}
\left\{    
    \begin{array}{ll}
    -\p_tw^i_\b(t,x)-\b\D w_\b^i(t,x)+H(x,D_xw_\b^i(t,x),X^{-i,\b}_t,\a^{-i,\b}(t,X^{-i,\b}_t))=0, & {\rm in}\ (0,T)\times\Rd,\\[3pt]
        \dot{X}^{i,\b}_t=\a^{i,\b}(t,X^{i,\b}_t), & {\rm in}\ (0,T),\\[3pt]
        \a^{i,\b}(t,x)=-D_pH(x,D_xw_\b^i(t,x),X^{-i,\b}_t,\a^{-i,\b}(t,X^{-i,\b}_t)), & {\rm in}\ (0,T)\times\Rd,\\[3pt]
        w^i_\b(T,x)=g(x,X^{-i,\b}_T),\ X^{i,\b}_0=z^i, & {\rm in}\ \Rd,\\
    \end{array}
\right.
\end{equation*}

We can rewrite the system above with a close relative of the mapping $\Phi$, which we denote as $\Phi^N:(\Rd)^N\times(\Rd)^N\rightarrow(\Rd)^N$, with the property that
\[\Phi^N(x,p)=-D_pH\left(x^i,p^i,m^{N,-i}_{x,\Phi^N(x,p)}\right).\]
Such a mapping $\Phi^N$ was proven to exist, is unique and Lipschitz in \cite[Lemma 3.5]{jackson2025quantitativeconvergencedisplacementmonotone}\footnote{The mapping $\Phi^N$ was denoted $a^N$ therein, but we do not use this notation to avoid confusion with $\a$.} whenever $N\in\N$ is large enough, we will assume that this is the case from now on and we rewrite the PDE system into
\begin{equation}\label{DNash intermediate}
\left\{    
    \begin{array}{ll}
    -\p_tw^i_\b(t,x)-\b\D w_\b^i(t,x)+H\left(x,D_xw_\b^i(t,x),X^{-i,\b}_t,\Phi^{N,i}\left(X^{\b}_t,D_xw_\b(t,X^\b_t)\right)\right)=0, & {\rm in}\ (0,T)\times\Rd,\\[3pt]
\dot{X}^{i,\b}_t=\Phi^{N,i}\left(X^{\b}_t,D_xw_\b(t,X^\b_t)\right), & {\rm in}\ (0,T),\\[3pt]
        w^i_\b(T,x)=g(x,X^{-i,\b}_T),\ X^{i,\b}_0=z^i, & {\rm in}\ \Rd.\\
    \end{array}
\right.
\end{equation}

We emphasise again that the states $X^\b$ are deterministic for all $\b\geq 0.$

\begin{lemma}\label{N vanishing}
There exists a classical solution $(w_\b,X^\b)$ to \eqref{DNash intermediate} for all $\b\geq0$. Moreover, for all $R>0$
    \[\lim_{\b\rightarrow0}\max_{i=1,\ldots,N}\left(\norm{w^i_\b-w^i}_{L^\i([0,T]\times B_R)}+\norm{D_xw^i_\b-D_xw^i}_{L^\i([0,T]\times B_R)}\right)=0,\]
    \[\lim_{\b\rightarrow0}\sup_{t\in[0,T]}|X^\b_t-X_t|=0\]
    %
    where we denoted $(w_0,X^0)$ as $(w,X)$.
\end{lemma}

\begin{proof}
The existence of solutions can be shown with the same method as in Theorem \ref{existence thm stocahstic}, we do not claim any uniqueness (we will not need uniqueness). Similarly to the proof of Theorem \ref{thm vanishing viscosity MFG}, by classical control estimates we can apply the Arzelà--Ascoli theorem, and we can extract a subsequence $\{\b_j\}_{j\in\N}$ such that  $w^i_{\b_j}$ converges to  some function $w^i\in C^{1,1}([0,T]\times\Rd)$ locally uniformly, with $D_xw^i_{\b_j}$ also converging locally uniformly to $D_xw^i.$ That is,
    \[\lim_{\b_j\rightarrow0}\max_{i=1,\ldots N}\Bigl(\norm{w^i_{\b_j}-w^i}_{L^\i([0,T]\times B_R)}+\norm{D_xw^i_{\b_j}-D_xw^i}_{L^\i([0,T]\times B_R)}\Bigr)=0.\]
    Denote by $X\in C([0,T];(\Rd)^N)$ the unique solution to the ODE system
\[ \left\{   
    \begin{array}{ll}
      \dot{X}^{i}_t=\Phi^{N,i}\left(X_t,D_xw(t,X_t)\right), & t\in (0,T)\\
        X^{i}_0=z^i.
    \end{array}
 \right.\]
Let $R>0$ be large enough so that $B_R$ contains the all the trajectories $X^i$, i.e. $\sup_{1\leq i\leq N}\norm{X^i}_\i\leq R$, then by Jensen's inequality and Lipschitz property of $\Phi^N$ from \cite[Lemma 3.5]{jackson2025quantitativeconvergencedisplacementmonotone},
\begin{align*}
    |X_t-X^{\b_j}_t|^2&\leq C\int^t_0 |X_s-X^{\b_j}_s|^2+\sum^N_{i=1}|D_xw^i(s,X^i_s)-D_xw^i_{\b_j}(s,X^{i,\b_j}_s)|^2\ ds\\
    &\leq C\left(\sup_{1\leq i\leq N}\norm{D_xw^i_{\b_j}-D_xw^i}^2_{L^\i([0,T]\times B_R)}+\int^t_0 |X_s-X^{\b_j}_s|^2 ds \right).
\end{align*}
Moreover, the constant $C>0$ is independent of $\b_j$. 

\medskip

By Grönwall's inequality we have $\sup_{t\in[0,T]}|X^{\b_j}_t-X_t|^2\rightarrow0$, as $j\to\infty$. Passing to the limit we find that $(w,X)$ is a solution to \eqref{DNash deterministic}, which admits unique solutions, therefore we have that necessarily the whole sequence converges.
\end{proof}
Finally, we arrive at the main theorem of this section.
\begin{theorem}\label{conv thm5}
     Suppose
     \begin{itemize}[label=\raisebox{0.25ex}{\tiny$\bullet$}]
     \item Assumption \ref{assumption MF} and \ref{Assumption MF semimon} hold true;
         \item $\b=0$;
         \item $(w,X,\a)$ is the solution to the $N$-player system \eqref{DNash deterministic} with initial conditions $z^i\in\Rd$;
         \item $(v,\m)$ is the solution to the MFG system \eqref{MFG} with initial condition $m_0\in\prob_q(\Rd),\ q>2,\ q\notin\{4,\frac{d}{d-2}\}$.
     \end{itemize} Then there exists a constant $C>0$ independent of $N\in\N$ such that
    \[\max_{i=1,\ldots,N}\norm{D_xw^i-D_xv}_\infty\leq C\left(K(N)+r_{d,q}(N)\right)^\frac{1}{2}\]
    where we recall the quantities $K(N),r_{d,q}(N)$ are defined as in the Theorem \ref{conv thm 1} and \ref{conv thm2}.
\end{theorem}

\begin{proof}
    For all $R,\b>0,1\leq i \leq N$, we have 
    \begin{equation*}
        \norm{D_xw^i-D_xv}_{L^\i([0,T]\times B_R)}\leq \norm{D_xw^i_\b-D_xw^i}_{L^\i([0,T]\times B_R)}+\norm{D_xw^i_\b-D_xv_\b}_{\i}+\norm{D_xv_\b-D_xv}_{L^\i([0,T]\times B_R)}.
    \end{equation*}
    where $w_\b$ and $v_\b$ solves \eqref{DNash intermediate} and \eqref{b-MFG} with $\b>0$ respectively.

    \medskip
    Let us now derive a bound on the middle term using the maximum principle. For some $j\in\{1,\ldots, N\}$ define $\varphi:=\p_{x_j}w^i_\b$ and $ \psi:=\p_{x_j}v_\b$. Then $\varphi$ and $\psi$ are classical solutions (regularity can be deduced by the bootstrapping method also used in the proof of Theorem \ref{existence thm stocahstic}) of the parabolic PDEs respectively
    \[\begin{cases}
        -\p_t\varphi-\b\D\varphi+\p_{x_j}H_{w}+D_pH_w D_x\varphi=0, &\textnormal{in } (0,T)\times\Rd,\\
        \varphi(T,x)=\p_{x_j}g\left(x,m^{N,-i}_{X^\b_T}\right), &\textnormal{in } \Rd.
    \end{cases}\]
\[\begin{cases}
    -\p_t\psi-\b\D\psi+\p_{x_j}H_{v}+D_pH_v D_x\psi=0, &\textnormal{in } (0,T)\times\Rd,\\
        \psi(T,x)=\p_{x_j}g\left(x,m_T^\b\right), &\textnormal{in } \Rd.
\end{cases}\]
where $H_{w}:=H\left(x,D_xw_\b^i,X^{-i,\b}_t,\a^{-i,\b}(t,X^{-i,\b}_t)\right)$ and $H_v:=H(x,D_xv_\b,\m^\b_t).$ 

\medskip

We have
\begin{equation*}
 \p_{x_j}H_{w}-\p_{x_j}H_{v}+D_x\varphi(D_pH_w-D_pH_v)\leq C\left(W_2\left(\m^\b_t,m^{N,-i}_{X^\b_t,A^\b_t}\right)+\norm{D_xw^i_\b(t,\cdot)-D_xv_\b(t,\cdot)}_\i\right)=:C_t
\end{equation*}
where we denoted $A^\b_t:=(\a^{i,\b}(t,X^{i,\b}_t))_{i=1,\ldots,N}$ and the constant $C>0$ depends only on the Lipschitz constants of $D_xH, D_pH$ and $D_xw_\b.$
Furthermore, as $\norm{D^2_{xx}w_\b}_{\op,\i}\leq C_1$ can be made independent of $\b$ and $N$, so is the constant $C>0$ above. Note that as $D_xw^i_\b,D_xv_\b\in L^\i$, the quantity $C_t$ is well defined.

\medskip

Similarly, there is constant $C>0$ independent of $\b,N$ such that
\[\p_{x_j}g\left(x,m^{N,-i}_{X^\b_T}\right)-\p_{x_j}g\left(x,m_T^\b\right)\geq -CW_2\left(m^\b_T,m^{N,-i}_{X^\b_T}\right):=-C_T.\]
Define $\Psi(t,x):=\varphi(T-t,x)-\psi(T-t,x)+C_T+\int^{T-t}_0 C_s\ ds$, considering the difference between the equations satisfied by $\varphi$ and $\psi$, the identity $D_x\varphi D_pH_w- D_x\psi D_pH_v=D_pH_v(D_x\varphi-D_x\psi)+D_x\varphi(D_pH_w-D_pH_v)$ and the estimates derived above, we find that $\Psi$ is a subsolution of the linear parabolic PDE
\[\begin{cases}
    \p_t\Psi-\b\D\Psi+D_pH_vD_x\Psi\leq0, &\textnormal{in } (0,T)\times\Rd,\\
    \Psi(0,x)\geq0 &\textnormal{in } \Rd.
\end{cases}\]
By the maximum principle \cite[Section 2.4, Theorem 9]{friedman2013partial}, $\Psi(t,x)\geq0$ for all $(t,x)\in[0,T]\times\Rd$, i.e.
\begin{equation*}
    \varphi(t,x)-\psi(t,x)\geq-\left(C_T+\int^{t}_0 C_s\ ds\right)\geq-\left(\widetilde{C}_T+C\int^t_0\norm{D_xw^i_\b(s,\cdot)-D_xv_\b(s,\cdot)}_\i\ ds\right)
\end{equation*}
where $\widetilde{C}_T:=C_T+C\int^T_0W_2\left(\m^\b_t,m^{N,-i}_{X^\b_t,A^\b_t}\right)\ dt$. By a symmetric argument, we can obtain the bound for all $j=1,\ldots,N$
\[|\p_{x_j}w^i_\b(t,x)-\p_{x_j}v_\b(t,x)|\leq\left(\widetilde{C}_T+C\int^t_0\norm{D_xw^i_\b(s,\cdot)-D_xv_\b(s,\cdot)}_\i\ ds\right),\]
squaring, using Jensen's inequality and then summing over $1,\ldots,N$ we further obtain
\[\norm{D_xw^i_\b(t,\cdot)-D_xv_\b(t,\cdot)}_\i^2\leq\left(2d^2\widetilde{C}_T^2+C\int^t_0\norm{D_xw^i_\b(s,\cdot)-D_xv_\b(s,\cdot)}^2_\i\ ds\right).\]

\medskip

By Grönwall's inequality, there is a constant $\widehat{C}>0$ independent of $\b,N$ such that
\begin{align*}
    \norm{D_xw^i_\b(t,\cdot)&-D_xv_\b(t,\cdot)}_\i\leq\widehat{C}\left(W_2\left(m^\b_T,m^{N,-i}_{X^\b_T}\right)+\int^T_0W_2\left(\m^\b_t,m^{N,-i}_{X^\b_t,A^\b_t}\right)\ dt\right)\\
    &\leq\widehat{C}\left(\sup_{t\in[0,T]}W_2\left(\m^\b_t,m^{N,\-i}_{\Xh_t,\Ah_t}\right)+\left(\frac{1}{N-1}\E\Biggl[\sup_{t\in[0,T]}|X^\b_t-\Xh_t|^2+\int^T_0|A^\b_t-\Ah_t|^2\ dt\Biggr]\right)^\frac{1}{2}\right),
\end{align*}
where $\Xh,\Ah$ are corresponding state and control processes of the FBSDE \eqref{iidPMPMFG} with $\b=0$. We note that $X^\b,A^\b$ are deterministic processes.

\medskip

Combining Theorem \ref{thm vanishing viscosity MFG}, Lemma \ref{N vanishing}, Theorem \ref{conv thm2} and \cite[Theorem 1]{Fournier2013OnTR}, we find
\begin{align*}
    \limsup_{\b\rightarrow0}\  \norm{D_xw^i_\b(t,\cdot)-D_xv_\b(t,\cdot)}_\i&\leq \widehat{C}\left(K(N)+r_{d,q}(N)\right)^\frac{1}{2},
\end{align*}
as the constant $\widehat{C}>0$ is independent of $R,\b>0$ and $N\in\N$, we can conclude the proof.    
\end{proof}
An immediate corollary is a refinement of the quantitative convergence of trajectories between the players of the fully deterministic $N$-player game and representative players of the deterministic MFG starting in the same initial conditions.

\medskip

To this end, if for some $1\leq i\leq N$, the point $z^i$ belongs to the support of the initial population $m_0$, we introduce the state process $\Xh^{i,z^i}$ of a representative player of the deterministic MFG starting at initial condition $z^i$ to be the first component of the solution to the FBODE
\begin{equation}\label{representative PMPMFG}
    \begin{cases}
        \Xh^{i,z^i}_t=z^i-\int^t_0 D_pH(\Xh^{i,z^i}_s,\Yh^{i,z^i}_s,\m_s) \ ds \\
        \Yh^{i,z^i}_t=D_xg(\Xh^{i,z^i}_T,m_T)-\int^T_tD_xH(\Xh^{i,z^i}_s,\Yh^{i,z^i}_s,\m_s)\ ds.     
    \end{cases}
\end{equation}
where $\Yh^{i,z^i}_t=D_xv(t,\Xh^{i,z^i}_t)$ and $(v,\m)$ is the solution to the deterministic MFG \eqref{MFG} with $\b=0$. 

\begin{corollary}\label{conv 4}
     Suppose that we are in the setting of the previous theorem. Let $1\leq i\leq N$, $z^i\in\spt(m_0),$ then there exists a constant $C>0$ independent of $N\in\N$ such that
    \[\sup_{t\in[0,T]}\left(|X^i_t-\Xh^{i,z^i}_t|^2+|A^i_t-\Ah^{i,z^i}_t|^2 \right)\leq C(K(N) +r_{d,q}(N))\]
    where 
    \begin{itemize}[label=\raisebox{0.25ex}{\tiny$\bullet$}]
        \item $X^i$ is the $i$-th coordinate of the first component of the solution to the FBODE \eqref{deterministic PMP} with initial conditions $(z^1,\dots,z^N)$;
        \item $\Xh^{i,z^i}$ is the $i$-th coordinate of the first component of the solution to the FBODE \eqref{representative PMPMFG};
        \item $A^i_t=\a^i(t,X^i_t)$, $A^{i,z^i}_t:=-D_pH(\Xh^{i,z^i}_t,\Yh^{i,z^i}_t,\m_t)$ are the corresponding control processes
    \end{itemize} and the quantities $K(N)$ and $r_{d,q}(N)$ are defined as in the Theorem \ref{conv thm 1} and \ref{conv thm2}.
    %
    %
\end{corollary}

\begin{proof}
Using Lipschitz properties of $D_pH$ and $D_xv$, we can obtain the bound below using Theorem \ref{conv thm5}
    \begin{align*}
        |X^i_t-\Xh^{i,z^i}_t|^2&\leq T\int^t_0 |A^i_s-\Ah^{i,z^i}_s|^2\ ds\\
        &\leq C\left( K(N)+ r_{d,q}(N)+\int^t_0 |X^i_s-\Xh^{i,z^i}_s|^2\ ds\right). 
    \end{align*}
    We can conclude the proof by Grönwall's inequality.
\end{proof}

\section{A Linear Quadratic Problem}\label{LQ section}

In the spirit of \cite{BardiLQMFG}, we derive explicitly computable linear quadratic (LQ) examples for $N$-player games and corresponding MFGs. To help with the ease of notations, we will reuse some notations for variables defined in previous sections.

\medskip
 
\begin{example}\label{ex:Npl}
Consider the $N$-player LQ problem in dimension one with $\b=0$, Dirac mass initial conditions (i.e. we are in the fully deterministic setting of \eqref{DNash deterministic}) and running costs
 \begin{equation}\label{LQ costs functions}L^i(a,A^{-i})=\frac{1}{2}\left(a+\frac{\k}{N-1}\sum_{j\neq i}A^j\right)^2+\frac{\g a^2}{2},\ g^i(x,y^{-i})=\frac{1}{2}\left(x+\frac{\varrho}{N-1}\sum_{j\neq i}y^j\right)^2,
 \end{equation}
where $\k,\varrho\neq0$ and  $\g>0$. Note that players are \say{penalised} according to how much they deviate from a factor of the mean velocity and final position (or the negative means when $\k,\varrho>0)$.

\medskip

Direct computations show that the functions defined in \eqref{LQ costs functions} satisfy Assumptions \ref{Assumption semimon} and \ref{assumption DxL}. In particular \eqref{C disp} holds with $C_{L,x}=0$, where
\[C_{L,a}:=\min\Biggl\{1+\g+\left(1-\frac{1}{N-1}\right)\k,1+\g-\frac{\k}{N-1}\Biggr\},\] 
\[C_g:=-\min\Biggl\{1+\g+\left(1-\frac{1}{N-1}\right)\varrho,1+\g-\frac{\varrho}{N-1}\Biggr\}.\]
Any solution found is in fact unique by Corollary \ref{uniqueness cor} if \eqref{C disp} holds.

\medskip

Furthermore, \eqref{LQ costs functions} satisfies Assumption \ref{Assumption L loc} whenever $|\k|<1+\g$, in which case we know that solutions exist by Theorem \ref{existence thm intermediate}.
\medskip

We find that the Hamiltonians are of the form
%
%
%
\[\Ht^i(p,A^{-i})=\frac{p^2}{2(\g+1)}+\frac{\k p}{\g+1}\frac{1}{N-1}\sum_{j\neq i} A^j-\frac{\g \k^2}{2(\g+1)}\left(\frac{1}{N-1}\sum_{j\neq i}A^j\right)^2.\]
In light of the LQ nature of our problem at hand, we make the following ansatz. 
\begin{ansatz}\
    \begin{enumerate}
        \item Affine feedback functions: $\a^i(t,x)=K_i(t)x+C_i(t)$
        \item Quadratic value functions: $w^i(t,x)=\frac{1}{2}r_i(t)x^2+p_i(t)x+q_i(t)$
    \end{enumerate}
    where $K_i,C_i,r_i,p_i,q_i$ are functions of time to be determined, from now on we will suppress the dependence on $t$ for ease of notation.
\end{ansatz}
    We first consider the consistency relation
    \begin{align*}
        K_ix+C_i=\a^i(t,x)=&-D_pH^i(D_xw^i(t,x),\a^{-i}(t,X^{-i}_t))\\
        =&-\frac{1}{\g+1}\left(D_xw^i(t,x)+\frac{\k}{N-1}\sum_{j\neq i}\a^j(t,X^j_t)\ \right)\\
        =&-\frac{r_i}{\g+1}x-\frac{p_i}{\g+1}-\frac{\k}{\g+1}\frac{1}{N-1}\sum_{j\neq i}(C_j+K_j X^j).
    \end{align*}
Therefore $K_i=-\frac{r_i}{\g+1}$ and $(\g+1)C_i=-p_i-\frac{\k}{N-1}\sum_{j\neq i}(C_j+K_j X^j)$. Summing the latter identity over $i=1,\ldots,N$ gives, for $\k\neq-(1+\g)$
\[\sum_{i=1}^N C_i=\frac{1}{\k+\g+1}\sum^N_{i=1}\left(-p_i+\frac{\k r_i X^i}{\g+1} \right)\]
and consequently, for $\k\neq (\g+1)(N-1)$
\[C_i=\frac{N-1}{(\g+1)(N-1)-\k}\left(-p_i-\frac{\k}{(\k+\g+1)(N-1)}\sum_{j=1}^N\left(-p_j+\frac{\k r_j X^j}{\g+1}\right)+\frac{\k}{(\g+1)(N-1)}\sum_{j\neq i}r_j X^j\right).\]
%
%
%
Considering the HJB equations gives the ODE system
\[\begin{cases}
    r_i'=\frac{r_i^2}{\g+1}\\
    p_i'+r_iC_i=0\\
     q_i'+\frac{p_i^2}{2}+(\g+1)p_iC_i+\frac{\g(\g+1)C_i^2}{2}=0.
\end{cases}\]
The terminal conditions for the ODEs above follows from 
\[g^i(x,X^{-i}_T)=\frac{1}{2}\left(x+\frac{\varrho}{N-1}\sum_{j\neq i}X^j_T\right)^2,\]
therefore $r_i(T)=1,\ p_i(T)=\frac{\varrho}{N-1}\sum_{j\neq i} X^i_T$ and $q_i(T)=\frac{1}{2}(\frac{\varrho}{N-1}\sum_{j\neq i}X^i_T)^2.$

\medskip

We readily solve the ODE for $r_i$,
\[r_i(t)=\frac{\g+1}{T+\g+1-t}.\]

The ODE system for $p_i$ is coupled with the state equations, which forms a fully coupled forward backward system
\begin{equation}\label{ODE}
    \begin{cases}
        p_i'+\frac{C_i}{T+1-t}=0 &p_i(T)=\frac{\varrho}{N-1}\sum_{j\neq i}\m_i(T)\\
        X'_i=-\frac{r_i }{\g+1}X_i+C_i &X_i(0)\in\R.
    \end{cases}
\end{equation}
In principle, this can be solved (even analytically) by summing \eqref{ODE} over $i=1,\ldots,N$ and solving first for the quantities $P(t):=\sum^N_{i=1}p_i(t),\ M(t):=\sum^N_{i=1} X_i(t).$ 
\end{example}

\begin{example}\label{ex:mfg}
Let us now turn to the corresponding MFG, the cost functions we have considered for the $N$-player games are the finite dimensional projections (i.e. \eqref{MF cost functions}) of 
\[L(a,\nu)=\frac{1}{2}\left(a+\k\int_\R A\ \nu(A)\right)^2+\frac{\g a^2}{2},\ g(x,m)=\frac{1}{2}\left(x+\varrho\int_\R y\ dm(y)\right)^2, \]
as before, we require $\g>0,\k\neq 0,-(1+\g).$

\medskip

%
%
%
%
  %
The corresponding Hamiltonian is 
\[H(p,\nu)=\frac{p^2}{2(\g+1)}+\frac{\k p}{\g+1} \int_\R A\ d\nu(A)-\frac{\g\k^2}{2(\g+1)}\left(\int_\R A\ d\nu(A)\right)^2.\]
The cost functions are displacement semimonotone in the sense of Assumption \ref{Assumption MF semimon}, provided
\begin{equation}\label{kappa rho condition}1+\k+\g+T(1+\varrho)>0.\end{equation}
Note that in this case the running cost (hence Hamiltonian) is independent of $x$, the consistency relation reduces to 
$$\nu_t=(-D_pH(D_xv(t,\cdot),\nu_t))_\sharp m_t.$$ 
\medskip

For the MFG we consider Gaussian initial conditions\footnote{Abusing notation to not distinguish between an absolutely continuous measure and its density.}, i.e. 
\[m(0,x)=\frac{s(0)}{\sqrt{\pi}}\exp\left(-\bigl(s(0)(x-\m(0))\bigr)^2\right)\]
for some $s(0)>0$ and $\m(0)\in\R$. Note that we will abuse notation and use the notation $m$ for both the measure and its density.

\medskip

We make the following ansatz (c.f. \cite{BardiLQMFG}). 

\begin{ansatz}\
    \begin{enumerate}
        \item Affine feedback functions: $\a(t,x)=K(t)x+C(t)$
        \item Gaussian flows: $m(t,x)=\frac{s(t)}{\sqrt{\pi}}\exp\left(-\bigl(s(t)(x-\m(t))\bigr)^2\right)$
        \item Quadratic value functions: $v(t,x)=\frac{1}{2}r(t)x^2+p(t)x+q(t)$
    \end{enumerate}
\end{ansatz}
We note that in light of the ansatz above, these problems are also known as linear quadratic Gaussian (LQG) problems in the literature.

\medskip

Firstly, we consider the consistency relation, which using the ansatz rewrites to $\nu_t=\a(t,\cdot)_\sharp m_t$ and
\begin{align*}
    Kx+C=\a(t,x)&=-D_pH(D_xv(t,x),\nu_t)\\
    &=-\frac{1}{\g+1}\left(D_xv(t,x)+\k\int_\R A\ d\nu_t(A)\right)\\
    &=-\frac{r}{\g+1}x-\frac{p}{\g+1}-\frac{\k}{\g+1}\int_\R \a(t,y)\ dm_t(y)\\
    &=-\frac{r}{\g+1}x-\frac{p}{\g+1}-\frac{\k}{\g+1}(K\m+C), 
\end{align*}
therefore $K=-\frac{r}{\g+1}$ and $C=\frac{1}{\k+\g+1}(-p +\frac{\k r\m}{\g+1})$. 

\medskip


Considering the continuity equation gives the ODE system 
\[\begin{cases}
    s'-\frac{sr}{\g+1}=0    \\
    \m'+\frac{\m r}{\g+1}-C=0     ,
\end{cases}\]
considering the HJB equation gives
\[\begin{cases}
    r'=\frac{r^2}{\g+1}\\
    p'+rC=0\\
    q'+\frac{p^2}{2}+(\g+1)pC+\frac{\g(\g+1)C^2}{2}=0.
\end{cases}\]
As in the $N$-player games,
\[r(t)=\frac{\g+1}{T+\g+1-t}\ \textnormal{and}\ s(t)=\frac{s(0)(T+\g+1)}{(T+\g+1-t)}.\]
We can solve analytically the coupled forward backward system for $p$ and $\m$
\begin{equation*}
    \begin{cases}
        p'-\frac{\g+1}{\k+\g+1}\frac{p}{T+\g+1-t}+\frac{\k(\g+1)}{\k+\g+1}\frac{\m}{(T+\g+1-t)^2}=0,   &p(T)=\varrho\m(T),\\
        \m'+\frac{1}{\k+\g+1}p+\frac{\g+1}{\k+\g+1}\frac{\m}{T+\g+1-t}=0,  &\m(0)\in\R.
    \end{cases}
\end{equation*}
Using Mathematica the solution takes the form
%
%
%
%
%
%
%
\[\m(t)=B-\frac{D(T+\g+1+t)}{\k},\ p(t)=-\frac{B(\g+1)}{T+\g+1-t}+D\left(\frac{2(\g+1)(\g+1+T)}{\k(T+\g+1-t)}+1\right)\]
where $B,D$ are constants satisfying
\[\begin{cases}
    B-\frac{T+\g+1}{\k}D=\m(0)\\
    -B+(\frac{2(\g+1+T)}{\k}+1) D=p(T)=\varrho\m(T)=\varrho\left(B-\frac{2T+\g+1}{\k}D\right)
\end{cases}\]
That is 
\begin{equation}\label{matrix}
        \begin{pmatrix}
    1 &  -\frac{T+\g+1}{\k}\\
    -1-\varrho & \frac{2(T+\g+1)+\varrho(2T+\g+1)}{\k}+1& \\
\end{pmatrix}\begin{pmatrix}
    B\\
    D\\
\end{pmatrix}=\begin{pmatrix}
    \m(0)\\0\\
\end{pmatrix}.
\end{equation}
Considering the determinant, we find that \eqref{matrix} has a unique solution if 
%
%
%
%
\[1+\k+\g+T(1+\varrho)\neq0,\]
which is guaranteed by the semimonotonicity condition \eqref{kappa rho condition}.

\medskip

Lastly, 
\[C(t)=\frac{B}{T+\g+1-t}-\frac{2D(T+\g+1)}{\k(T+\g+1-t)},\]
\[q(t)=-\frac{2BD(\g+1)(T+\g+1)}{T+\g+1-t}+\frac{2D^2(\g+1)(T+\g+1)^2}{\k^2(T+\g+1-t)}+\frac{B^2(\g+1)}{2(T+\g+1-t)}-\frac{D^2t}{2}+E\]
where $E$ is a constant such that $q(T)=\frac{1}{2}(\varrho \m(T))^2$.
One can also recover $\nu_t$ by computing \[\nu_t=\left(-\frac{r(t)}{\g+1}\textnormal{Id}+C(t)\right)_\sharp m_t.\]
\end{example}

To see the sharpness of the contractivity condition 9 in Assumptions \ref{Assumption L loc} and \ref{assumption MF}, based on the previous examples, we can formulate the following observation.

\begin{proposition}\
\begin{itemize}[label=\raisebox{0.25ex}{\tiny$\bullet$}]
\item  Suppose that in the case of Example \ref{ex:mfg} the condition $\k\neq -(1+\g)$ is violated. Then condition 9 in Assumption \ref{assumption MF} is not fulfilled and the corresponding MFG does not possess any quadratic solutions.

\item Similarly, suppose that the condition 
$\k\neq(\g+1)(N-1)$ in 
Example \ref{ex:Npl} is violated. Then condition 9 in Assumption \ref{Assumption L loc} is not fulfilled and the corresponding $N$-player game does not possess any quadratic solutions.
\end{itemize}
\end{proposition}


Regarding the sharpness of the displacement semimonotonicity, based again on the previous examples, we can formulate the following observation.

\begin{proposition}
    Let $ 1+\k+\g+T(1+\varrho)=0$, $\m(0)=0$. Then the displacement semimonotonicity condition is violated and the corresponding LQ MFG has infinitely many solutions.
\end{proposition}
\begin{proof}
    We can readily check the violation of the displacement semimonotonicity condition. Furthermore, the system \eqref{matrix} has infinitely many solutions in this case. 
\end{proof}

\appendix

\section{Properties of the Legendre Transform}

Let $\mathscr{X}$ be a Banach space, $L\in C^1(\mathscr{X}\times\Rd)$ be strictly convex in the second variable and satisfy the uniform coercivity condition
\begin{equation}\label{coercive}L(y,a)\geq \th(|a|)-C\ \ \forall (y,a)\in\mathscr{X}\times\Rd\end{equation}
for some superlinear function $\th:[0,\i)\rightarrow[0,\i)$ and constant $C\geq 0$.

\medskip

This result is essentially the same as \cite[Theorem A.2.6 (a)]{CanSin}, for completeness we provide a proof.
%
%
\begin{theorem}\label{legendre}
    Under the condition \eqref{coercive}, for all $R>0$, there exists $C_R$ such that $|D_pH(y,p)|<C_R$ for all $(y,p)\in \mathscr{X}\times B_R$, where $H$ is the associated Hamiltonian
\[H(y,p)=\max_{a\in\Rd}\{-p\cdot a-L(y,a)\}\]
\end{theorem}

\begin{proof}
Firstly, we note that by \eqref{coercive}
\[\lim_{|a|\rightarrow\i}\inf_{y\in\mathscr{X}} \frac{L(y,a)}{|a|}\rightarrow+\i.\]
By a subdifferential expansion, we have 
\[\inf_{y\in\mathscr{X}}\frac{L(y,0)}{|a|}\geq \inf_{y\in\mathscr{X}}\frac{L(y,a)}{|a|}-\inf_{y\in\mathscr{X}}D_aL(y,a)\cdot\frac{a}{|a|},\]
taking $|a|\rightarrow\i$, we find that we must have $\lim_{|a|\rightarrow\i}\inf_{y\in\mathscr{X}}|D_aL(y,a)|=+\i$.

\medskip

    Suppose that there is $R>0$, $(y_n,p_n)\in \mathscr{X}\times B_R$ such that $q_n:=-D_pH(y_n,p_n)$ satisfies $|q_n|\rightarrow+\i$ as $n\rightarrow\i$, up to extracting a subsequence we can assume $p_n\rightarrow p\in B_R$. The first order optimality condition reads
    \[p_n=-D_aL(y_n,q_n),\]
    then
    \[|p_n|\geq \inf_{y\in\mathscr{X}}|D_aL(y,q_n)|\]
    %
    %
    Taking $n\rightarrow\i$ we have $|p|\geq +\i$, which is a contradiction.
\end{proof}

Alternatively, we can also prove the following result under a relaxed coercivity condition.
\begin{equation}\label{coercivity relax}
    L(y,a)\geq \th(|a|)-C(1+\norm{y}_\mathscr{X})
\end{equation}

\begin{theorem}\label{Legendre thm 2}
    Under assumption \eqref{coercivity relax}, for all $R>0$, there exists $C_R$ such that $|D_pH(y,p)|<C_R$ for all $(y,p)\in B_R$, where $B_R$ is now a ball in the Banach space $\mathscr{X}\times\Rd.$
\end{theorem}

\begin{proof}
     Suppose that there is $R>0$, $(y_n,p_n)\in B_R$ such that $q_n:=-D_pH(y_n,p_n)$ satisfies $|q_n|\rightarrow+\i$ as $n\rightarrow\i$, up to extracting a subsequence we can assume $(y_n,p_n)\rightarrow (y,p)\in B_R$. 

     \medskip
     
    Then, as $q_n$ maximises the Hamiltonian at $(y_n,a_n)$, we have
    \begin{align*}-L(y_n,0)
    &\leq H(y_n,p_n)=-q_n\cdot p_n-L(y_n,q_n)\\
    &\leq -q_n\cdot p_n-\th(|q_n|)+C\norm{y_n}_\mathscr{X}\\
    &\leq|q_n|\left(|p_n|-\frac{\th(|q_n|)}{|q_n|}\right)+C(1+\norm{y_n}_{\mathscr{X}}).\end{align*}
    
    Taking $n\rightarrow\i$ we have $-L(y,0)\leq -\i$, which is a contradiction.
\end{proof}

\medskip

\begin{remark}\label{coercivity remark}
    Under Theorem \ref{Legendre thm 2}, \cite[Theorem A. 2.6. (b),(c), Corollary A.2.7]{CanSin} are still true. 
%
%
Moreover, one can carefully modify the proof of \cite[Theorem 7.4.6.]{CanSin} to the case of \eqref{coercivity relax}, so that the results we have used in Chapter 7 of the same reference are still valid.
\end{remark}

\section{Convexity of Value Functions}

We were unable to find a reference in the stochastic control literature for the claim that the value function of a stochastic control problem inherits convexity in space from convex cost functions. However the proof is essentially the same as in the deterministic case, see \cite[Theorem 7.4.13]{CanSin}.
%

\begin{theorem}\label{convexity thm}
Given a stochastic optimal control problem with running cost $L:[0,T]\times\Rd\times\Rd\rightarrow\R$ jointly convex in the latter two variables and convex terminal cost $g:\Rd\rightarrow\R$. 

\medskip

The value function given by
\[v(t,x)=\inf_\a\E\left[\int^T_t L(s,Y_s,\a_s)\ ds+g(Y_T)\right]\]
    where the infimum is taken over all (admissible) $\sF$-progressively measurable processes $\a$ and
    \[dY_s=\a_sds+\sqrt{2\b}dB_s,\ \ Y_t=x,\]
    is convex in the spatial variable
\end{theorem}

    \begin{proof}
        Let $\e>0,\ \l\in(0,1),\ i=1,2$, $x^i\in\Rd$, $\a^i$ be $\e$-optimal controls for the initial positions $x^i$ respectively and $Y^i$ be the corresponding trajectories respectively, i.e.
        \[\E\left[\int^T_tL(s,Y^i_s,\a^i_s)\ dt+g(Y^i_T)\right]\leq v(t,x^i)+\e.\]
        Note that $\l\a^1+(1-\l)\a^2$ is an admissible control at the initial position $\l x^1+(1-\l)x^2$. Furthermore, by (strong) uniqueness its corresponding trajectory is given by 
        \[Y^\l:=\l Y^1+(1-\l)Y^2.\]
        Therefore by the convexity of the cost functions
        \begin{align*}
            v(t,\l x^1+(1-\l)x^2)&\leq\E\left[\int^T_tL(s,Y_s^\l,\l\a_s^1+(1-\l)\a_s^2)\ dt+g(Y^\l_T)\right]\\
            &\leq \l v(t,x^1)+(1-\l) v(t,x^2)+2\e.
        \end{align*}
        As $\e>0$ was arbitrary, the proof is complete.
    \end{proof}

  \section{Wellposedness of Nonlinear Nonlocal Continuity Equations}

The following is a variant of \cite[Lemma 3.6]{alparmouMFGunderDmonotone} in the setting of $W_2$ instead of $W_1$.

\begin{theorem}\label{thm cont eqn}
    Let $T>0$ and $V:[0,T]\times\Rd\times\prob_2(\Rd)\rightarrow\Rd$ be a given continuous vector field and suppose that there exists $C_V>0$ such that $x\mapsto V(t,x,\m)$ is $C_V$-Lipschitz uniformly in the variables $(t,\m)$ and
    \[|V(t,0,\m)|\leq C_V,\ \ |V(t,x,\m)-V(t,x,\nu)|\leq C_V W_2(\m,\nu) \ \ \ \ \forall t\in[0,T],\m,\nu\in\prob_2(\Rd).\]
    Then, for any $m_0\in\prob_q(\Rd)$, $q>2$, the nonlinear nonlocal continuity equation 
    \[\p_tm_t+\na\cdot(m_tV(t,x,m_t))=0\]
    with initial condition $m_0$ admits a unique distributional solution $m\in C([0,T];(\prob_2(\Rd),W_2))$.
\end{theorem}

\begin{proof}[Sketch of proof]
 The proof follows the same ideas as in the fixed point scheme in the proof of \cite[Lemma 3.6]{alparmouMFGunderDmonotone}, we only indicate where changes are needed. 
    
    \medskip
    
    We define the mapping $S$ therein on the space $C([0,T];(\prob_2(\Rd),W_2))$, for ease to the reader we reproduce here the definition of $S$, for $\widetilde{m}\in C([0,T];(\prob_2(\Rd),W_2))$, we define $S(\widetilde{m})=m$ where $m$ is the unique distributional solution to the linear continuity equation
    \[\p_tm_t+\na\cdot(m_tV(t,x,\widetilde{m}_t))=0\]
    with initial condition $m_0$.

    \medskip

    The range of $S$ is relatively compact with respect to the topology of uniform convergence in $ C([0,T];(\prob_2(\Rd),W_2))$ due to the following moment estimate, which can be derived with similar computation as in \cite[Claim 1]{alparmouMFGunderDmonotone},
    \[\int_\Rd |x|^q dm_t\leq C\left(1+\int_\Rd |x|^q dm_0\right)\]
    for some constant $C>0$ depending only on $C_V$ and $T$ and the following equicontinuity estimate
    \begin{align*}
       W_2^2(m_t,m_s)\leq C|t-s|
    \end{align*}
    for some constant $C>0$ depending only on $C_V, T$ and $\int_\Rd |x|^q dm_0$, which can be derived by  the SDE representation (not to be confused with the flow map denoted by $X(t,\cdot)$ used in \cite{alparmouMFGunderDmonotone})
    \[X_t=\xi+\int^t_0 V(s,X_s,\widetilde{m}_s)ds,\ \Leb(\xi)=m_0,\ \Leb(X_t)=m_t.\]
    The continuity of $S$ can be inferred from classical SDE stability \cite[Theorem 3.2.4]{Zhang2017}, we hence deduce the existence of solutions by Schauder fixed point theorem. The uniqueness follows from \cite[Theorem 3.3]{GangboSwiech2014}.
\end{proof}


\bibliographystyle{abbrv}
\bibliography{bibliography}

\end{document}